\documentclass[12pt]{amsart}
\usepackage{amsmath,amsfonts,amsthm,amscd,amssymb,stmaryrd,skull}
\newtheorem{theorem}{Theorem}
\newtheorem{claim}[theorem]{Claim}
\newtheorem{proposition}[theorem]{Proposition}
\newtheorem{lemma}[theorem]{Lemma}
\newtheorem{definition}[theorem]{Definition}

\newtheorem{corollary}[theorem]{Corollary}
\newtheorem{remark}[theorem]{Remark}
\numberwithin{equation}{section}
\numberwithin{theorem}{section}

\newcommand{\RR}{\mathbb{R}}
\newcommand{\CP}{\mathbb{CP}}

\newcommand{\ck}{\mathcal{K}}
\newcommand{\tr}{\mbox{tr}_g }

\newcommand{\n}{\nabla}

\newcommand{\mf}{\mathcal{F}}
\newcommand{\tmf}{\tilde{\mathcal{F}}}

\begin{document}
\bibliographystyle{amsalpha}
\title[Rigidity and Stability]{Rigidity and stability of Einstein metrics for quadratic
curvature functionals}
\author{Matthew J. Gursky}
\address{Department of Mathematics, University of Notre Dame, Notre Dame, IN 46556}
\email{mgursky@nd.edu}
\author{Jeff A. Viaclovsky}
\address{Department of Mathematics, University of Wisconsin, Madison, WI 53706}
\email{jeffv@math.wisc.edu}
\thanks{The first author has been partially supported under NSF grants DMS-0800084
and DMS-1206661.
The second author has been partially supported under NSF grants DMS-0804042
and DMS-1105187.}

\begin{abstract}
We investigate rigidity and stability properties
of critical points of quadratic curvature functionals
on the space of Riemannian metrics. We show it is possible to ``gauge'' the
Euler-Lagrange equations, in a self-adjoint fashion,
to become elliptic. Fredholm theory may then be used to describe
local properties of the moduli space of critical metrics. We show a number
of compact examples are infinitesimally rigid, and
consequently, are isolated critical points in the space of
unit-volume Riemannian metrics.
We then give examples of critical metrics which are strict local
minimizers (up to diffeomorphism and scaling).
A corollary is a local ``reverse Bishop's inequality''
for such metrics.  In particular,
any metric $g$ in a $C^{2,\alpha}$-neighborhood of
the round metric $(S^n,g_S)$ satisfying $Ric(g) \leq Ric(g_S)$ has volume $Vol(g) \geq Vol(g_S)$, with equality holding if and
only if $g$ is isometric to $g_S$.
\end{abstract}

\maketitle
\setcounter{tocdepth}{1}
\vspace{-1mm}
\tableofcontents
\section{Introduction}

\subsection{Quadratic Functionals}  Let $M$ be a closed manifold of dimension
$n \geq 3$.  The total scalar curvature, or Einstein-Hilbert
functional, has been deeply studied in Riemannian geometry and we do not
attempt to give a survey here. We only remark that Einstein metrics in
general are saddle points for the Einstein-Hilbert functional,
and rigidity and stability properties of Einstein metrics have been studied in
\cite{Koiso1, Koiso2, Koiso3}. In this article we are interested in
functionals on the space of Riemannian metrics $\mathcal{M}$ which are
quadratic in the curvature; see \cite{Besse, Blair, Smolentsev} for surveys. Such functionals have also been widely studied in
physics under the name ``fourth-order,'' ``critical,'' or ``quadratic'' gravity;
see for example \cite{LuPope, Maldacena, Schmidt, Stelle}.

Using the standard decomposition of the curvature tensor $Rm$ into the Weyl, Ricci and scalar curvature
curvature components, a basis for the space of quadratic curvature
functionals is
\begin{align} \label{FgenN}
\mathcal{W} = \int |W|^2\ dV, \ \ \rho = \int |Ric|^2\ dV,
\ \ \mathcal{S} = \int R^2\ dV.
\end{align}
In dimension three, $\mathcal{W}$ vanishes,
and in dimension four, the Chern-Gauss-Bonnet formula
\begin{align*}
32\pi^2 \chi(M) = \int |W|^2\ dV - 2 \int |Ric|^2\ dV + \frac{2}{3} \int R^2\ dV
\end{align*}
implies that $\mathcal{W}$ can be written as a linear combination of the
other two (plus a topological term). In higher dimensions, the Euler-Lagrange
equations of $\mathcal{W}$ have a different structure than
the latter two functionals.
Consequently, in this paper, we will be interested in the functional
\begin{align} \label{Ftdef1}
\mathcal{F}_{\tau}[g] = \int |Ric|^2\ dV + \tau \int R^2\ dV,
\end{align}
(with $\tau = \infty$ formally corresponding to $\int R^2 dV$)
in dimensions $n \geq 3$.
In dimensions other than four, the functional $\mathcal{F}_{\tau}$ is not scale-invariant.
Therefore, we will consider the volume-normalized
functional
\begin{align} \label{TFtdef}
\tilde{\mathcal{F}}_{\tau}[g] = Vol(g)^{\frac{4}{n}-1} \mathcal{F}_{\tau}[g].
\end{align}
Any Einstein metric is critical for $\tmf_{\tau}$
(this is not true for
$\tilde{\mathcal{W}}$; an additional algebraic condition on
the curvature tensor is required).

\subsection{Rigidity}
As with any geometric variational problem we begin by trying to understand the local properties of the moduli space of critical points.  This is the space of metrics $g$ satisfying
\begin{align} \label{gradzed}
\nabla \tmf_{\tau}[g] = 0
\end{align}
modulo diffeomorphism equivalence and scaling,
where $\nabla \tmf_{\tau}$ is the gradient of $\tmf_{\tau}$.   Due to diffeomorphism invariance,
the condition (\ref{gradzed})
does not define an elliptic equation in $g$.  In order to provide a reasonable
description of the moduli space near a critical metric, in Section \ref{Moduli}
we define a map
$P_g : \overline{S}_0^2(T^{*}M) \rightarrow \overline{S}_0^2(T^{*}M)$, where
\begin{align}
\label{S20}
\overline{S}_0^2(T^{*}M) = \Big\{ h \in S^2(T^{*}M) \ :\ \int (tr_g\ h)\ dV_g = 0
\Big\}.
\end{align}

In Section \ref{Moduli}, we will show that the zeroes of $P_g$
correspond exactly to critical metrics in a neighborhood of
$g$, up to diffeomorphism and scaling.
Furthermore, $P_g$ has the advantange of being elliptic.
This property guarantees that the kernel of the linearized operator,
denoted $H^1_{\tau}$, is finite dimensional.  Therefore, by using a standard
implicit function theorem argument one can define the {\em Kuranishi map}
$\Psi : H^1_{\tau} \rightarrow H^2_{\tau}$,
for which the moduli space is locally given by the orbit space
of the isometry group acting on the zero locus $\Psi^{-1}(0)$.
Note that here $H^2_{\tau}$ denotes the cokernel, but since our linearized
operator is self-adjoint this is the same as $H^1_{\tau}$ on a compact manifold.

\begin{definition}{\em
Let $M$ be a smooth, closed manifold of dimension $n$, and $g$ a critical metric for $\tmf_{\tau}$.
The metric $g$ is called {\em infinitesimally rigid} (for
$\tmf_{\tau}$) if $H^1_{\tau} = \{ 0 \}$,
and $g$ is called {\em rigid} (for $\tmf_{\tau}$) if there exists a
$C^{4, \alpha}$ neighborhood $U$ of $g$ in the
space of Riemannian metrics such
that if $\tilde{g} \in U$ and $\tilde{g}$ satisfies \eqref{gradzed},
then $\tilde{g} = e^c \phi^{*}g$ for some
$C^{5, \alpha}$-diffeomorphism $\phi : M \rightarrow M$
and a constant $c \in \mathbb{R}$.

  A non-zero element $h \in H^1_{\tau}$ is called {\em{integrable}}
there exists a path $g_s$ of critical metrics
(for $\tmf_{t}$) with $g_0 = g$ satisfying
$( g_s)' |_{s=0} = h$.
Otherwise, $h$ is called {\em{non-integrable}}.
}
\end{definition}

In Section \ref{LocalSec}, we will show that infinitesimal
rigidity implies rigidity.
Furthermore, it can indeed happen that there are non-trivial
elements in $H^1_{\tau}$ which are non-integrable.
This in fact occurs on $(S^3, g_S)$, where $g_S$ is the round
metric, see Theorem \ref{nonit} below.
A similar phenomenon occurs for infinitesimal Einstein deformations
on certain Einstein manifolds, the simplest known example
being $S^2 \times \CP^{2m}$ with $m \geq 2$ \cite{Koiso3}.

On an Einstein manifold, the Lichnerowicz Laplacian is given by
\begin{align*}
\Delta_L h_{ij} = \Delta h_{ij} + 2 R_{ipjq} h^{pq} - \frac{2}{n}Rh_{ij},
\end{align*}
and $\mbox{spec}_{TT} ( - \Delta_L)$ will denote
the set of eigenvalues of $(-\Delta_L)$ restricted to
transverse-traceless (TT) tensors. Infinitesimal Einstein rigidity is the condition
\begin{align*}
  \frac{2}{n}R
\notin \mbox{spec}_{TT} ( - \Delta_L),
\end{align*}
see \cite[Chapter 12]{Besse}. In particular,
if $g$ fails to be
rigid as an Einstein metric then it is not rigid as a critical point of
$\tmf_{\tau}$, but the converse may not be true.

 In Section \ref{RigidSec} we prove a vanishing result for $H^1_{\tau}$ when $g$
is an Einstein metric.
The full statement is rather complicated, since it depends on the
value of $\tau$ and the dimension
$n$; therefore, we only state here the result for the functional $\tmf_0 = Vol^{\frac{4}{n} -1} \int |Ric|^2 dV$
(which will be of particular interest for the applications described in Section \ref{RBIsec}):

\begin{theorem} \label{RigidMain}
Let $(M,g)$ be an $n$-dimensional Einstein manifold, $n \geq 3$.
When $R > 0$, assume $(M,g)$ is not isometric to the
round sphere. If
\begin{align} \label{notspec}
\Big\{ \frac{2}{n} R,   \frac{4}{n} R  \Big\}
\notin \mbox{spec}_{TT} ( - \Delta_L),
\end{align}
then $H^1_0 = \{0 \}$
and consequently $g$ is rigid for $\tmf_0$.

The same result holds when $R < 0$, assuming $n =3$ or $n =4$.
\end{theorem}

 On the round sphere $(S^n, g_S)$,
$H^1_0 = \{ f g : f \in \Lambda \}$,
where $\Lambda$  is the space of first-order spherical harmonics,
and these elements are generated by the action of the conformal group.
Consequently, $g_S$ is in fact rigid, see Section~\ref{Moduli}.

 The criterion in Theorem \ref{RigidMain} only involves the
action of the Lichnerowicz Laplacian on transverse-traceless variations of the
metric.  In fact, one also has to consider conformal variations, but rigidity
under these variations will follow from the Lichnerowicz estimate for the first
eigenvalue of the Laplacian (if $g$ is not isometric
to the round sphere), see Section~\ref{RigidSec}.  In the negative
case there is no a priori lower bound for the spectrum of $(-\Delta)$, as can
 be seen by considering hyperbolic examples with long
``necks''.  Consequently, in dimensions $n \geq 5$, our rigidity result for negative Einstein metrics
require an additional assumption on $\lambda_1(-\Delta)$;
see Lemma \ref{H1TT} and Theorem \ref{h1thm} for precise statements.

\subsection{Stability and local minimization}
We will also consider local minimizing properties of critical points,
and the relation with stability, which we now define.
\begin{definition}{\em
Let $M$ be smooth, closed manifold of dimension $n$, and $g$ a critical metric for $\tmf_{\tau}$.
Then $g$ is called a
{\em{local minimizer}} (for $\tmf_{\tau}$) if for all metrics
$\tilde{g}$ in a $C^{2,\alpha}$-neighborhood of $g$,
\begin{align} \label{gEminloc}
\tmf_{\tau}[\tilde{g}] \geq \tmf_{\tau}[g].
\end{align}
If, in addition, equality holds if and only if $\tilde{g} = e^c \phi^{*}g$ for some
$C^{3, \alpha}$-diffeomorphism $\phi : M \rightarrow M$ and a
constant $c \in \mathbb{R}$,  then $g$ is called a
{\em{strict local minimizer}}.

 A metric $g$ is called {\em{strictly stable}}
if  the Jacobi operator is positive (as
a bilinear form on $L^2$) when restricted to variations $h$ with $\delta_g h = \frac{1}{n}d(tr\ h)$.  If the Jacobi operator on this subspace
is positive for all $h$ orthogonal to the kernel and, in addition, all kernel elements are integrable,
then $g$ is called {\em{stable}}. }
\end{definition}
In Section \ref{StableSec} we show that if
$g$ is strictly stable then $g$ is a strict local
minimizer. Furthermore, if $g$ is stable, then $g$ is a
local minimizer, see Proposition~\ref{locmin}.
In this latter statement, the integrability
of the kernel directions is crucial, as this
implication is not necessarily true otherwise, see Theorem~\ref{nonit} below.

Again, the full statement of our main theorem
regarding stability for $\tmf_{\tau}$ is rather complicated,
so for the purpose of simplicity in the introduction, we only
state here the special case of $\tau > -1/n$:

\begin{theorem} \label{StableMain}  Let $(M,g)$ be an $n$-dimensional
Einstein manifold, $n \geq 3$, and let $\tau > -1/n$. If $R > 0$ and
\begin{align} \label{stablegap}
\mbox{spec}_{TT}(-\Delta_L) \cap \Big[ \frac{2}{n}R ,
\Big(\frac{4}{n} + 2\tau \Big) R \Big] = \varnothing,
\end{align}
then $g$ is a strict local minimizer for $\tmf_{\tau}$.
If $R < 0$ and $n = 3$ or $n = 4$, then the same result
holds provided that endpoints in \eqref{stablegap} are reversed.
\end{theorem}

 Strict stability of $g$ should imply that the gradient flow for the
functional will converge to $g$ for all initial
metrics sufficiently close to $g$. However, we will not pursue this
here, and refer the reader to \cite{Bour, Streets1, Streets2, Zheng, Ye}
for other results regarding gradient flows for Riemannian functionals.

\subsection{The Bach tensor in dimension four}
\label{Bachsub}

 We remark that for $\tau = -n/4(n-1)$, this functional is equivalent to
$\int \sigma_2(A_g) dV$, where $\sigma_2(A)$ denotes the second elementary symmetric
function of the Schouten tensor.
The conformal properties of this functional
have been widely studied, see \cite{Viaclovsky} for references.
Critical points of this functional in dimension three were studied by the authors in
in \cite{GurskyViaclovsky}.

In dimension four ($\tau = -1/3$) this functional is
conformally invariant, and a critical metric satisfies the Euler-Lagrange
equations
\begin{align}
\label{Bintro}
B_{ij} \equiv -4 \Big( \nabla^k \nabla^l W_{ikjl} + \frac{1}{2} R^{kl}W_{ikjl} \Big)= 0.
\end{align}
The tensor $B_{ij}$ is known as the Bach tensor \cite{Bach, Derdzinski}.
Any Einstein metric is necessarily Bach-flat.
The Bach tensor is conformally invariant, traceless, and divergence-free.

The linearized Bach tensor fails to be
elliptic due to both conformal and diffeomorphism invariance,
so the condition \eqref{Bintro}
does not define an elliptic equation in $g$. In Section \ref{BachSec}
we define a map
$P_g^B : S_0^2(T^{*}M) \rightarrow S_0^2(T^{*}M)$, where
$S_0^2(T^{*}M)$ consists of pointwise traceless tensors.
In Section \ref{BachSec}, we will show that the zeroes of $P_g^B$
correspond exactly to Bach-flat metrics in a neighborhood of
$g$, up to diffeomorphism and conformal change of metric.
Furthermore, $P_g^B$ has the advantange of being elliptic,
so there is again a Kuranishi map
$\Psi_B : H^1_{-1/3} \rightarrow H^2_{-1/3}$,
and the moduli space is locally described by the orbit
space of the isometry group acting on the zero locus $\Psi_B^{-1}(0)$.

 Due to this extra conformal invariance, the above conditions of rigidity and
strict local minimization need to be modified slightly for the Bach tensor:
\begin{definition}
\label{Bachdef}
{\em
Let $(M^4,g)$ be a Bach-flat manifold.
The metric $g$ is called {\em infinitesimally Bach-rigid} if $H^1_{-1/3} = \{ 0 \}$,
and $g$ is called {\em Bach-rigid} if there exists a $C^{4, \alpha}$ neighborhood $U$ of $g$ in the
space of Riemannian metrics such
that if $\tilde{g} \in U$ and $\tilde{g}$ satisfies \eqref{Bintro},
then $\tilde{g} = e^f \phi^{*}g$ for some
$C^{5, \alpha}$-diffeomorphism $\phi : M \rightarrow M$
and a $C^{4, \alpha}$ function $f: M \rightarrow \mathbb{R}$.

The Bach-flat metric $g$ is called a {\em{local minimizer for}} $\mathcal{W}$
if for all metrics $\tilde{g}$ in a $C^{2,\alpha}$-neighborhood of $g$,
\begin{align}
\label{Wineq}
\mathcal{W}[\tilde{g}] \geq \mathcal{W}[g].
\end{align}
If, in addition, equality holds in \eqref{Wineq}
if and only if $\tilde{g} = e^f \phi^{*}g$ for some
$C^{3, \alpha}$-diffeomorphism $\phi : M \rightarrow M$ and
$C^{2, \alpha}$-function $f : M \rightarrow \mathbb{R}$,
then $g$ is a {\em{strict local minimizer for}} $\mathcal{W}$.

 A metric $g$ is called {\em{strictly Bach-stable}}
if  the Jacobi operator is positive (as
a bilinear form on $L^2$) when restricted to transverse-traceless variations $h$.
If the Jacobi operator on $TT$-variations
is positive for all $h$ orthogonal to the kernel and, in addition,
all kernel elements are integrable, then $g$ is called {\em{Bach-stable}}. }
\end{definition}

In Section~\ref{StableSec} we show that if
$g$ is strictly Bach-stable then $g$ is a strict local
minimizer of $\mathcal{W}$. Furthermore, if $g$ is Bach-stable, then $g$ is a
local minimizer, see Proposition~\ref{Bachlocmin}.
Our main result regarding the rigidity and stability of Bach-flat
metrics is the following.
\begin{theorem}
\label{Bachthm}
Let $(M^4,g)$ be an Einstein manifold. Assume that
\begin{align} \label{notspecB}
\Big\{ \frac{1}{3} R,   \frac{1}{2} R  \Big\}
\notin \mbox{spec}_{TT} ( - \Delta_L).
\end{align}
Then $H^1_{-1/3} = \{0\}$ and consequently $g$ is Bach-rigid.

If $ R > 0 $, and $g$ moreover satisfies
\begin{align} \label{stablegapB}
\mbox{spec}_{TT}(-\Delta_L) \cap \Big[ \frac{1}{3} R, \frac{1}{2}R \Big] = \varnothing,
\end{align}
then $g$ is strict local minimizer for $\mathcal{W}$.
The same result holds if $R < 0$, provided that the
endpoints of the interval in \eqref{stablegapB} are reversed.
\end{theorem}

We remark that the second variation of $\mathcal{W}$ was previously computed at
an Einstein metric in dimension four by Kobayashi \cite{Kobayashi}.
The infinitesimal rigidity and stability statements follow from this
computation. Our main contribution in Theorem \ref{Bachthm}
is therefore to conclude actual rigidity and strict local minimization
from the infinitesimal statements.

\subsection{Examples}
We remark that the functional $\mathcal{R} = \int |Rm|^2 dV$
has been studied by many
authors, see for example \cite{AndersonI, AndersonII, Berger,
LeBrunOptimal, Muto, Streets1}.
Decomposing the curvature tensor, we may write
\begin{align} \label{RMF}
\mathcal{F}_{-1/2(n-1)} = \frac{n-2}{4} \{ \mathcal{R} - \mathcal{W} \}.
\end{align}
Consequently, if one can show that a metric $g$ is a local minimizer for $\tau = -1/2(n-1)$,
then if $g$ is also known to minimize $\tilde{\mathcal{W}}$ it follows from (\ref{RMF}) that
$\tilde{\mathcal{R}}$ is also locally minimized. Thus the above results
will also have applications to the variational structure of $\tilde{\mathcal{R}}$
in many examples.

In Section \ref{ExampleSec}, we will present examples for which rigidity
and stability hold. The following is a brief summary. \vskip.1in

\noindent
$\bullet$ The round sphere $(S^n, g_S)$:
\begin{theorem}
\label{snit}
On $(S^n, g_S)$, or any constant curvature quotient thereof,
if $n \geq 4$,  $g_S$ is a strict local minimizer
for $\tmf_{\tau}$ provided that
\begin{align}
 \frac{4 - 3n}{2 n (n-1)} < \tau < \frac{2}{n(n-1)}.
\end{align}
If $n = 3$, the same conclusion holds provided that
\begin{align}
- \frac{3}{8} < \tau < \frac{1}{3}.
\end{align}
For any $n \geq 3$, $g_S$ is a strict local minimizer for $\tilde{\mathcal{R}}$.
\end{theorem}
\noindent
We remark that the infinitesimal
stability for $\tilde{\mathcal{R}}$ was previously shown in \cite{Muto}.

 In the case of $n =3$, there is an interesting variation of the
round metric given by scaling the fibers of the Hopf fibration
$S^1 \rightarrow S^3 \rightarrow S^2 = \CP^1$, known as
{\em{Berger spheres}}.
In Section \ref{ExampleSec}, we will employ this variation to show that the
upper endpoint in Theorem \ref{snit} is sharp:
\begin{theorem}
\label{nonit}
For $n =3$, $g_S$ is {\em{not}} a local minimizer for $\tau = 1/3$, that is,
there is a path $g_s$ with $g_1 = g_S$, with
$\tmf_{1/3}[g_s] < \tmf_{1/3}[g_S]$ for $s < 1$.
Furthermore, there are non-trivial transverse-traceless
elements $h \in H^1_{1/3}$ which are not integrable.
\end{theorem}

As mentioned above, a similar non-integrability phenomenon occurs
for infinitesimal Einstein deformations on certain Einstein
manifolds \cite{Koiso3}. Our proof is similar to Koiso's in spirit: we
show that the {\em{third}} derivative of the functional (in our case
for the Berger sphere variation) is non-zero.

\vskip.1in
\noindent
$\bullet$  Compact hyperbolic manifolds:
\begin{theorem}
\label{hypint}
Let $(M^n,g)$ be a compact hyperbolic manifold.
If $n =3$ or $n =4$,  then $g$ is a strict local minimizer
for $\tmf_{\tau}$ provided that
\begin{align}
- \frac{1}{3} < \tau.
\end{align}
Consequently, $g$ is a strict local minimizer for  $\tilde{\mathcal{R}}$.
If $n \geq 5$,
then $g$ is a strict local minimizer for $\tmf_{\tau}$ provided that
\begin{align}
\frac{4 - 3n}{2 n (n-1)} < \tau \leq - \frac{1}{n}.
\end{align}
\end{theorem}

\vskip.1in

\noindent
$\bullet$ Complex projective space $(\CP^m, g_{FS})$, $m \geq 2$:

\begin{theorem}
\label{cpmint}
On  $(\CP^m, g_{FS})$, $m \geq 2$, the Fubini-Study metric $g_{FS}$ is a strict
local minimizer for
$\tmf_{\tau}$ provided that
\begin{align}
 \frac{2 - 3m}{2 m (2m-1)} < \tau < \frac{1}{m(m+1)}.
\end{align}
For $n = 2m = 4$, and $\tau = -1/3$, $g_{FS}$ is the unique
global minimizer of $\mathcal{W}$ among the class of metrics with
positive scalar curvature, up to diffeomorphism and
conformal changes of metric. Furthermore, in this case, $g_{FS}$
is a strict local minimizer for $\mathcal{R}$.
\end{theorem}
\noindent
The global minimization statement for $\mathcal{W}$ in dimension four
follows from the signature theorem together with a result of
Hitchin \cite{Besse}.

\vskip.1in
\noindent
$\bullet$ The product of round spheres $(S^m \times S^m, g_1 + g_2)$
with $m \geq 2$:

\begin{theorem}
\label{smsmint}
 On $(S^m \times S^m, g_1 + g_2)$ with $m > 1$, $g_1$ and $g_2$ round
metrics with $Vol(g_1) = Vol(g_2)$,
the product metric $g_1 + g_2$ is a strict local minimizer for $\tmf_{\tau}$ provided that
\begin{align}
\label{smsmr}
 \frac{2 - 3m}{2 m (2m-1)} < \tau < \frac{2-m}{2m(m-1)}.
\end{align}
For $n = 2m = 4$ and $\tau = -1/3$, $g_1 + g_2$ is the unique
global minimizer of $\mathcal{W}$ among the class of positive
scalar curvature metrics, up to diffeomorphism and
conformal changes of metric. Furthermore, in this case,
$g_1 + g_2$ is a strict local minimizer for $\mathcal{R}$.
\end{theorem}
\noindent
The global minimization statement for $\mathcal{W}$ was
proved in \cite{Gursky1998}.
We emphasize that $\tau = 0$ is not in the included range for $m =2$.
This is because there are non-trivial infinitesimal
deformations for $\tau= 0$ at the product metric, in fact,
$\dim(H^1_0) = 9$, see Section \ref{pmsm}.
We do not know if the right endpoint in \eqref{smsmr}
is optimal for strict minimization.

\vskip.1in
\noindent
$\bullet$ The Ricci-flat case:\vskip.1in

Ricci-flat metrics which admit nonzero
parallel spinors are locally maximizing in the transverse-traceless
direction as critical points of the Einstein-Hilbert functional
\cite{DWW1}. For quadratic functionals, we have the following.
\begin{theorem}
\label{rfint}
Let $(M,g)$ be a compact Ricci-flat manifold, and
assume that all infinitesimal Einstein deformations are
integrable. Then $g$ is a local minimizer for $\tmf_{\tau}$ provided that
\begin{align}
- \frac{n}{4(n-1)} < \tau.
\end{align}
In particular, this holds if $(M,g)$ is a flat torus, or if
$(M,g)$ is a Calabi-Yau metric.
\end{theorem}
Obviously, such a metric is minimizing if $\tau \geq - 1/n$, so the
main point is that this can be improved, at least locally.
The integrability of infinitesimal Einstein deformations on
a Calabi-Yau manifold was proved by Bogomolov and Tian
\cite{Bogomolov, TianCY}.

\subsection{Reverse Bishop's inequalities}  \label{RBIsec}
The classical Bishop's inequality implies that if $(M,g)$ is a closed manifold with
$Ric(g) \geq Ric(S^n) = (n-1)g$, then the volume satisfies
$Vol(g) \leq Vol(S^n)$, and equality holds only if $(M,g)$ is
isometric to the sphere.   An interesting consequence of stability for $\tau =0$
for Einstein metrics is that, locally, a ``reverse Bishop's inequality'' holds:

\begin{theorem} \label{RevBishop}  Let $(M,g)$ be an Einstein manifold with
positive scalar curvature,
normalized so that $Ric(g) = (n-1)g$.
Assume $g$ is a strict local minimizer for $\tmf_0$.  Then there exists a
$C^{2,\alpha}$-neighborhood $U$ of $g$ such that if $\tilde{g} \in U$ with
$ Ric(\tilde{g}) \leq (n-1)\tilde{g}$,
then $Vol(\tilde{g}) \geq Vol(g)$
with equality if and only if $\tilde{g} = \phi^{*}g$ for some
diffeomorphism $\phi : M \rightarrow M$.
\end{theorem}

From the list of examples above, we see that if $(M,g)$ is a sphere, space form, or complex projective space, then for metrics $\tilde{g}$ near $g$
\begin{align*}
Ric(\tilde{g}) \leq Ric(g) \Longrightarrow Vol(M,\tilde{g}) \geq Vol(M,g).
\end{align*}

Based on the above results for spheres and complex
projective spaces, we conjecture that
if $(M^n,g)$ is a compact rank one
symmetric space, then $g$ is a strict local minimizer for $\tmf_0$.
This would imply that Theorem \ref{RevBishop}
also holds for quaternionic projective spaces and the Cayley plane.

Theorem \ref{RevBishop} has a counterpart for negative Einstein manifolds,
see Corollary~\ref{RicRigidNeg}. We mention that on certain negative Einstein
manifolds, such as hyperbolic or K\"ahler-Einstein manifolds, there are
important global results which in some cases are much stronger than what our
local analysis can yield.  One example is the work of Besson-Courtois-Gallot
(\cite{BCG1}, \cite{BCG2}) on volume rigidity:  If $(M^n, g_0)$ is a compact
hyperbolic manifold, and $g$ is any Riemannian metric on $M^n$ then
\begin{align*}
Ric(g) \geq -(n - 1)g  \Longrightarrow Vol(g) \geq Vol(g_0),
\end{align*}
with equality if
and only if $g$ is isometric to $g_0$.
Note the requirement that $M^n$ support a hyperbolic metric: in particular,
this gives a global result for any negative Einstein three-manifold,
but does not give (even local) information in
dimensions $n \geq 4$ when $M^n$ is only known to admit an
Einstein metric.
We also mention the article \cite{DWW2}. In this work,
a local volume comparison theorem is proved using only
the scalar curvature, at a negative K\"ahler-Einstein metric.

\subsection{Global minimization}

The functionals $\int |Ric|^2$ and $\int R^2$ are known to
be globally minimized in dimension four at a negative K\"ahler-Einstein
metric. This is proved in \cite{LeBrunRCMV} using Seiberg-Witten
theory, along with many other interesting results regarding minimal volumes.

In any odd dimension, the fibration
$S^1 \rightarrow S^{2n +1} \rightarrow \CP^n$ has
totally geodesic fibers.  Thus by scaling the fiber, there are metrics
which collapse with bounded curvature on odd-dimensional
spheres \cite[Chapter 11]{Berger}, where it is also
pointed out that the global minimum of
$\tilde{\mathcal{R}}$ is equal to zero in dimensions $n \geq 5$
on {\em{any}} Riemannian manifold.
Thus the spherical metric is most definintely not a global minimizer
of $\tmf_0$ in dimensions other than four.
However, we conjecture that the round sphere $(S^4, g_S)$
is a {\em{global}} minimizer of $\mf_0$.

\subsection{Other functionals}

Many of the results in this paper hold for arbitrary
Riemannian functionals, e.g., Proposition \ref{locmin}.
For example, one can also consider other $L^p$-norms of the curvature.
The scale-invariant $L^{n/2}$-norm of the curvature
carries topological information, due to the Chern-Gauss-Bonnet Theorem.
We point out that many of the above local minimization theorems imply
local minimization for the
scale-invariant norm simply by H\"older's inequality when $n \geq 4$.
For example, for the functional $\tmf_0$,
\begin{align}
\tmf_0[g] = (Vol)^{4/n - 1} \int |Ric|^2 dV
\leq  \Big\{ \int |Ric|^{n/2} dV \Big\}^{4/n}.
\end{align}
Thus a local minimizer for $\tmf_0$ is automatically
a local minimizer for the scale-invariant norm of the Ricci tensor.
Thus many of the above results can be restated for the scale-invariant
power (but we do not explicitly restate each case separately here).
\subsection{Acknowledgements}
The authors would like to thank Gang Tian for helpful remarks regarding
the various ``slicing'' lemmas in the paper.  Thanks are also due to
Robert Bryant, Claude LeBrun, and Peter Li for helpful comments.  

Finally, the authors would also thank the referee, whose careful reading of the manuscript 
resulted in a number of improvements to the exposition.  

\section{Local properties of the moduli space}  \label{LocalSec}

In this Section we recall some basic variational formulae,
all of which are found in \cite[Chapter 4]{Besse}.
We begin with the squared $L^2$-norm of Ricci,
\begin{align}
\mathcal{F}_0 [g] = \int |Ric|^2 dV = \int g^{ip} g^{jq} R_{pq} R_{ij} dV.
\end{align}
Along an arbitrary path $g_s$, let $g_s' = \frac{d}{ds}g_s$.  We have
\begin{align}
\frac{d}{ds} \mathcal{F}_0 [g]
 =  \int  g^{ip} g^{jq}(g_s')_{ij}
 ( \nabla \mathcal{F}_0)_{pq}  dV,
\end{align}
where
\begin{align}
(\nabla \mathcal{F}_0)_{pq} =  -\Delta_L (Ric)_{pq}  + \nabla_p \nabla_q R
- \frac{1}{2}(\Delta R) g_{pq} - 2 R^l_p R_{lq} + \frac{1}{2} |Ric|^2 g_{pq}.
\end{align}
Recalling that
\begin{align}
(\Delta_L h)_{ij} = (\Delta h)_{ij} + 2R_{ipjq} h^{pq} - R_i^l h_{lj} -  R_j^l h_{il},
\end{align}
the gradient is also written as
\begin{align} \label{gradformF}
(\nabla \mathcal{F}_0)_{pq} =
- \Delta (Ric)_{pq} - 2 R_{pkql} R^{kl} + \nabla_p \nabla_q R
- \frac{1}{2}(\Delta R) g_{pq} + \frac{1}{2} |Ric|^2 g_{pq}.
\end{align}
\begin{proposition}
\label{Ricc}
If $g$ is Einstein with $Ric = (R/n) g$, then
\begin{align}
 \nabla \mathcal{F}_0 [g] = \frac{n-4}{2n^2} R^2 g.
\end{align}
\end{proposition}
\begin{proof}
A simple computation from the above formula.
\end{proof}
Next, we consider the functional
\begin{align}
\mathcal{S}[g] = \int R^2 dV.
\end{align}
We have
\begin{align}
\frac{d}{ds} \mathcal{S}[g] = \int  g^{ip} g^{jq}(g_s')_{ij}
(\nabla \mathcal{S} )_{pq}  dV,
\end{align}
where
\begin{align} \label{gradformS}
( \nabla \mathcal{S} )_{pq} =  2 \n_p \n_q R - 2 (\Delta R) g_{pq}
- 2 R R_{pq} +   \frac{1}{2} R^2 g_{pq}.
\end{align}
\begin{proposition}
\label{Rc}
 If $g$ is Einstein with $Ric = (R/n) g$, then
\begin{align}
 \nabla \mathcal{S} [g] = \frac{n-4}{2n} R^2 g.
\end{align}
\end{proposition}
\begin{proof}
Another simple computation.
\end{proof}

\subsection{Volume normalized functionals}

The functional $\mathcal{F}_{\tau}$ is not scale-invariant
in dimensions other than four. A standard remedy for
this is to normalize the functionals by a suitable power of
the volume, which is equivalent to restricting the
functionals to the space of unit volume metrics $\mathcal{M}_1$.
Therefore, we need to understand the effect of the normalizing
term on the formulas for the first and second variation of $\mathcal{F}_{\tau}$.
The following is elementary:
\begin{claim}  The functional
\begin{align}
\tmf_{\tau}[g] = (Vol(g))^{\frac{4}{n} -1} \cdot \mf_{\tau}[g]
\end{align}
is scale-invariant.
\end{claim}
In the following, we let $p = \frac{4}{n} - 1$.
Let $g_s$ denote a path of metrics through $g = g_0$.

\begin{lemma} The first and second variation of the volume are
\begin{align}
(Vol(g_s))' &= \frac{1}{2} \int_M (tr_{g_s} g_s') dV\\
\label{vdd}
(Vol(g_s))'' &= - \frac{1}{2} \int_M  | g_s'|^2 dV
+ \frac{1}{2} \int_M ( tr_{g_s} g_s'') dV + \frac{1}{4}
\int ( tr_{g_s} g_s')^2 dV.
\end{align}
\end{lemma}
\begin{proof}
See \cite{Besse}, Chapter 4.
\end{proof}

Next, we consider the first variation of $\tmf_{\tau}$:

\begin{lemma}
A critical metric of $\tmf_{\tau}$ satisfies
\begin{align}
\label{foc}
\mf_{\tau}' &= - p (Vol)^{-1} (Vol)' \cdot \mf_{\tau},\\
\label{foc2}
\nabla \mf_{\tau}  &= -  \frac{p}{2} Vol^{-1} \cdot \mf_{\tau} \cdot g_s,
\end{align}
where $\nabla \mf_{\tau}$ denotes the gradient of $\mf_{\tau}$,
\end{lemma}
\begin{proof}
Differentiating the functional, we compute
\begin{align}
\label{oned}
(\tmf_{\tau}[g_s])' = p (Vol)^{p-1} (Vol)' \mf_{\tau} + (Vol)^p (\mf_{\tau}[g_s])'.
\end{align}
Therefore a critical point of $\tmf_{\tau}$ clearly satisfies \eqref{foc}.
We also can write this as
\begin{align}
\begin{split}
(\tmf_{\tau}[g_s])' &=  p (Vol)^{p-1} \mf_{\tau} \cdot \frac{1}{2} \int_M (tr_{g_s} g_s') dV
+ (Vol)^p \int_M \langle \nabla \mf_{\tau}, g_s' \rangle dV\\
& = Vol^p \Big(
\int \langle  \nabla \mf_{\tau} + \frac{p}{2} Vol^{-1} \cdot \mf_{\tau} \cdot g_s, h \rangle dV \Big).
\end{split}
\end{align}
Consequently, a critical point of $\tmf_{\tau}$ also satisfies \eqref{foc2}.
\end{proof}

\subsection{Local structure of the moduli space}
\label{Moduli}

To understand to local properties of the moduli space of
critical metrics, we need to cast the equation in the
framework of Fredholm theory. This requires an elliptic
operator, so one must ``gauge'' the equation to become
elliptic.

We begin with some definitions. For simplicity of notation, we will treat
the domain and range of an operator as if it were the bundle itself, although
the operator really acts on sections of the bundle.
Let $\delta_g : S^2(T^{*}M) \rightarrow T^{*}M$ denote the divergence operator
\begin{align} \label{deldef}
(\delta_g h)_j = \nabla^i h_{ij},
\end{align}
and $\delta^{*} : T^{*}M  \rightarrow S^2(T^{*}M)$ its $L^2$-adjoint.  Note that
\begin{align} \label{deldual}
\delta^{*} = -\frac{1}{2} \mathcal{L},
\end{align}
where $\mathcal{L}$ is the Killing operator:
\begin{align} \label{Ldef}
(\mathcal{L}_g \omega)_{ij} = \nabla_i \omega_j + \nabla_j \omega_i.
\end{align}
We let $\mathcal{K}_g$ denote the conformal Killing operator, the trace-free part of $\mathcal{L}_g$:
\begin{align} \label{Kef}
(\mathcal{K}_g \omega)_{ij} = \nabla_i \omega_j + \nabla_j \omega_i - \frac{2}{n}(\delta_g \omega)g_{ij}.
\end{align}
Note that if
\begin{align} \label{betadef}
\beta_g h = \delta_g h - \frac{1}{n} d(tr\ h),
\end{align}
then
\begin{align*}
\beta^{*}  = -\frac{1}{2} \mathcal{K}.
\end{align*}

Let $g$ be a critical metric for the normalized
functional
\begin{align} \label{normalF2}
\tilde{\mathcal{F}}_{\tau} = (Vol)^{\frac{4}{n}-1} \int \big( |Ric|^2 + \tau R^2 \big)\ dV.
\end{align}
Assume $\mathcal{U} \subset S^2(T^*M)$ is a neighborhood of the zero section, sufficiently small so that $\theta \in \mathcal{U}_0 \Rightarrow \tilde{g} = g + \theta$ is a metric.
Consider the map $P_g : \mathcal{U} \rightarrow S^2(T^{*}M)$ given by
\begin{align} \label{Pdef}
P_g(\theta) = \nabla \tilde{\mathcal{F}}_{\tau} (g + \theta) + \frac{1}{2} \mathcal{K}_{g + \theta}[ \beta_g \mathcal{K}_g \beta_g \theta]  - \gamma(\theta)\cdot g,
\end{align}
where
\begin{align} \label{gammadef}
\gamma(\theta) = \frac{1}{n} Vol(g)^{-1} \int tr_g \big\{ \nabla \tilde{\mathcal{F}}_{\tau}(g+\theta) \big\}\ dV_g.
\end{align}
The term involving $\gamma$ is included to make the trace of $P$ have mean value zero (with respect to $g$); hence
\begin{align*}
P_g : \mathcal{U}_0  \rightarrow \overline{S}_{0}^2(T^{*}M),
\end{align*}
where
\begin{align*}
\mathcal{U}_0 = \mathcal{U} \cap \overline{S}_{0}^2(T^{*}M),
\end{align*}
and $\overline{S}_{0}^2(T^{*}M)$ is defined in \eqref{S20}.
The domain of $P_g$ will be tensors in $\mathcal{U}_0$ of H\"older class $C^{k + 1, \alpha}$.
The term involving the operator $\mathcal{K}$ is included to handle gauge-invariance.  To simplify notation, we denote
\begin{align} \label{BoxLdef}
\Box_{\mathcal{K}_g} \omega &= \beta_g \mathcal{K}_g \omega
= -\frac{1}{2} \mathcal{K}_g^{*} \mathcal{K}_g \omega,
\end{align}
so that the gauge-fixing term can be written
\begin{align} \label{GFterm}
\frac{1}{2} \mathcal{K}_{g + \theta}[ \beta_g \mathcal{K}_g \beta_g \theta] = \frac{1}{2} \mathcal{K}_{g + \theta} [ \Box_{\mathcal{K}_g} \beta_g \theta].
\end{align}
Note that the kernel of $\Box_{\mathcal{K}_g}$ is precisely the space of conformal Killing forms (with respect to $g$), since
\begin{align} \label{BoxLKer}
\langle \omega, \Box_{\mathcal{K}_g} \omega \rangle_{L^2} = -\frac{1}{2} \| \mathcal{K}_g \omega \|_{L^2}^2.
\end{align}
\begin{lemma}
\label{kellip}
The operator $\Box_{\mathcal{K}_g}: T^*M \rightarrow T^*M$ is elliptic
for $n \neq 2$.
\end{lemma}
\begin{proof}
Since $\mathcal{K}_g \omega$ is traceless, we have
\begin{align}
( \Box_{\mathcal{K}_g} \omega)_j = (\Delta \omega)_j + \nabla^i \nabla_j \omega_i
- \frac{2}{n} \nabla_j \nabla^i \omega_i.
\end{align}
Consequently, the symbol is given by
\begin{align}
(\xi, \omega) \mapsto |\xi|^2 \omega_j + \left(1 - \frac{2}{n} \right)
\xi_j ( \xi_i \omega_i).
\end{align}
Pairing this with $\xi_j$ shows injectivity for $\xi \neq 0$.
\end{proof}
The main result of this section is that the moduli space of critical metrics is precisely given by the zero-locus of $P_g$:

\begin{theorem}  \label{Pzed}  Assume
\begin{align} \label{badt}
\tau \neq -\frac{n}{4(n-1)},
\end{align}
and consider $P_g$ as a mapping from $C^{k+1, \alpha}$ to $C^{k-3,\alpha}$ for $k \geq 3$.
Then: \vspace{1mm}

\noindent $(i)$ The linearization of $P_g$ at $\theta = 0$ is an elliptic operator.
\vspace{1mm}

\noindent $(ii)$ If $\theta$ is sufficiently small and
$P_g(\theta) = 0$, then the metric $\tilde{g} = g + \theta$ is a
smooth Riemannian metric and is
critical for the normalized functional $\tilde{\mathcal{F}}_{\tau}$. \vspace{1mm}

\noindent $(iii)$  Conversely, if $g_1 = g + \theta_1$ is a critical metric in a sufficiently small $C^{k+1, \alpha}$-neighborhood of $g$ ($k \geq 3$), then there exists a $C^{k+2,\alpha}$-diffeomorphism $\phi : M \rightarrow M$ and a constant $c$ such that
\begin{align} \label{thtildef}
e^c \phi^{*}g_1 = g + \tilde{\theta}
\end{align}
with
\begin{align} \label{Pcon}
P_g(\tilde{\theta}) = 0
\end{align}
and
\begin{align}
\label{trcon}
\int tr_g \tilde{\theta}\ dV_g = 0.
\end{align}
\end{theorem}

\begin{remark} {\em When $\tau = \tau_c =  -\frac{n}{4(n-1)}$, then
\begin{align} \label{s2}
\tilde{\mathcal{F}}_{\tau_c}[g] = -2(n-2)^2 Vol(g)^{\frac{4}{n}-1} \int \sigma_2(A_g)\ dV_g.
\end{align}
When $n \neq 4$, a metric $g$ is critical over all
conformal variations which preserve the volume if and only if
\begin{align*}
\sigma_2(A_g) = const.,
\end{align*}
a second order condition in the metric (for an explanation of the notation, see Section \ref{Bachsub} in the Introduction). When $n=4$, the integral in (\ref{s2}) is conformally invariant (hence the gradient is trace-free).  In both cases, this reduction in the order of the Euler equation results in a degeneracy on the symbol level.}
\end{remark}

\begin{remark}{\em
 The choice of the ``gauge-fixing'' term is not unique.  For the linearized Einstein equations, one can choose a gauge-fixing term which
reduces the symbol to that of the rough Laplacian (see \cite{GrahamLee}).  Such a simplification does not seem possible for the linearized gradient of (\ref{normalF2}).}
\end{remark}

\begin{proof}[Proof of Theorem \ref{Pzed}]
We will postpone the proof of ellipticity, as it involves a rather lengthy calculation of the symbol.  To prove part $(ii)$,
suppose
\begin{align} \label{zedP}
0 = P_g(\theta) =  \nabla \tilde{\mathcal{F}}_{\tau} (g + \theta) + \frac{1}{2} \mathcal{K}_{g + \theta}[\Box_{\mathcal{K}_g} \beta_g  \theta]  - \gamma(\theta)\cdot g.
\end{align}
Let $\tilde{g} = g + \theta$.    If we trace (\ref{zedP}) with
respect to $\tilde{g}$ and integrate,
\begin{align} \label{zedP1} \begin{split}
0 &= \int tr_{\tilde{g}} [\nabla \tilde{\mathcal{F}}_{\tau}(\tilde{g})]\ dV_{\tilde{g}} +  \frac{1}{2} \int tr_{\tilde{g}} \big[\mathcal{K}_{\tilde{g}}[ \Box_{\mathcal{K}_g} \beta_g \theta]\big]\ dV_{\tilde{g}} -  \gamma(\theta) \int tr_{\tilde{g}} g\ dV_{\tilde{g}}  \\
&= \int tr_{\tilde{g}} [\nabla \tilde{\mathcal{F}}_{\tau}(\tilde{g})]\ dV_{\tilde{g}}  - \gamma(\theta) \int tr_{\tilde{g}} g\ dV_{\tilde{g}}.
\end{split}
\end{align}
Here we used the fact that $\mathcal{K}[\cdot]$ is trace-free.
Also, since $\tilde{\mathcal{F}}_{\tau}$ is scale-invariant,
\begin{align*}
\tilde{\mathcal{F}}_{\tau}[ \tilde{g}] = \tilde{\mathcal{F}}_{\tau}[ e^r \tilde{g}], \ r \in \mathbb{R}.
\end{align*}
Differentiating this identity at $r = 0$,
\begin{align}  \label{trGr} \begin{split}
0 &= \frac{d}{dr} \tilde{\mathcal{F}}_{\tau}[ e^r \tilde{g}] \big|_{r=0}
= \int \big \langle \nabla \tilde{\mathcal{F}}_{\tau}(\tilde{g}),
\frac{d}{dr} \Big( e^r \tilde{g} \Big)\big|_{r=0} \big \rangle\ dV_{\tilde{g}} \\
&= \int \big \langle \nabla \tilde{\mathcal{F}}_{\tau}(\tilde{g}), \tilde{g}  \big \rangle\ dV_{\tilde{g}}
= \int tr_{\tilde{g}} [\nabla \tilde{\mathcal{F}}_{\tau}(\tilde{g})]\ dV_{\tilde{g}}.
\end{split}
\end{align}
Therefore, by (\ref{zedP1}) we conclude
\begin{align*}
\gamma(\theta) \int tr_{\tilde{g}} g\ dV_{\tilde{g}} = 0.
\end{align*}
Since $g$ is positive-definite,
\begin{align*}
\int tr_{\tilde{g}} g\ dV_{\tilde{g}} > 0,
\end{align*}
we conclude that $\gamma(\theta) = 0$.  Therefore, by the definition of $P$ in (\ref{Pdef}),
\begin{align} \label{zedP2}
0 =  \nabla \tilde{\mathcal{F}}_{\tau} (g + \theta) + \frac{1}{2} \mathcal{K}_{g + \theta}[ \Box_{\mathcal{K}_g} \beta_g \theta].
\end{align}

The next step involves an integration by parts argument, but this presents a difficulty since $\theta \in C^{4, \alpha}$ only implies that
$P_g(\theta)$ is $C^{\alpha}$, and not necessarily differentiable.  To get around this problem we mollify $\theta$; i.e., let $\{ \theta_{\epsilon}\}$ be a family
of smooth tensor fields such that $\theta_{\epsilon} \rightarrow \theta$ in $C^{4,\alpha}$ as $\epsilon \rightarrow 0$, and let $\tilde{g}_{\epsilon} = g + \theta_{\epsilon}$.  From (\ref{zedP2}) and the  continuity of $P$
it follows that
\begin{align} \label{zedPe}
\eta_{\epsilon} =  \nabla \tilde{\mathcal{F}}_{\tau} (g + \theta_{\epsilon}) + \frac{1}{2} \mathcal{K}_{g + \theta_{\epsilon}}[ \Box_{\mathcal{K}_g} \beta_g \theta_{\epsilon}],
\end{align}
where $\eta_{\epsilon} \rightarrow 0$ in $C^{\alpha}$.  Pair both sides of (\ref{zedPe}) with $\mathcal{L}_{\tilde{g_{\epsilon}}}[\Box_{\mathcal{K}_g} \beta_g \theta_{\epsilon}] $ (with respect to the $L^2$-inner product defined by $\tilde{g_{\epsilon}}$), where $\mathcal{L}$ is the
Killing operator defined in (\ref{Ldef}):
\begin{align*}
\langle \mathcal{L}_{\tilde{g}_{\epsilon}}[\Box_{\mathcal{K}_g} \beta_g \theta_{\epsilon}] , \eta_{\epsilon}\rangle_{L^2} &= \big \langle \mathcal{L}_{\tilde{g}_{\epsilon}}[\Box_{\mathcal{K}_g} \beta_g \theta_{\epsilon}] , \nabla \tilde{\mathcal{F}}_{\tau} (\tilde{g}_{\epsilon}) + \frac{1}{2} \mathcal{K}_{\tilde{g}_{\epsilon}}[ \Box_{\mathcal{K}_g} \beta_g \theta_{\epsilon}]  \big \rangle_{L^2} \\
&= \big \langle \mathcal{L}_{\tilde{g}_{\epsilon}}[\Box_{\mathcal{K}_g} \beta_g \theta_{\epsilon}] , \nabla \tilde{\mathcal{F}}_{\tau} (\tilde{g}_{\epsilon}) \big \rangle_{L^2} + \frac{1}{2} \| \mathcal{K}_{\tilde{g}_{\epsilon}} [\Box_{\mathcal{K}_g} \beta_g \theta_{\epsilon}] \|_{L^2}^2.
\end{align*}
Integrating by parts in the first term on the right-hand side, we obtain
\begin{align*}
\big \langle \mathcal{L}_{\tilde{g}_{\epsilon}}[\Box_{\mathcal{K}_g} \beta_g \theta_{\epsilon}] , \nabla \tilde{\mathcal{F}}_{\tau} (\tilde{g}_{\epsilon}) \big \rangle_{L^2} &= - 2 \big \langle \Box_{\mathcal{K}_g} \beta_g \theta_{\epsilon} , \delta_{\tilde{g}_{\epsilon}} \big( \nabla \tilde{\mathcal{F}}_{\tau} (\tilde{g}_{\epsilon}) \big) \big \rangle_{L^2} = 0,
\end{align*}
since $\mathcal{L}^{*} = -2 \delta$ and the gradient of a Riemannian functional is always divergence-free (see \cite{Besse}, Proposition 4.11).
Therefore,
\begin{align*}
\langle \mathcal{L}_{\tilde{g}_{\epsilon}}[\Box_{\mathcal{K}_g} \beta_g \theta_{\epsilon}] , \eta_{\epsilon}\rangle_{L^2} = \frac{1}{2} \| \mathcal{K}_{\tilde{g}_{\epsilon}} [\Box_{\mathcal{K}_g} \beta_g \theta_{\epsilon}] \|_{L^2}^2.
\end{align*}
Letting $\epsilon \rightarrow 0$, the left-hand side converges to zero, while the right-hand side converges to $\mathcal{K}_{\tilde{g}}[\Box_{\mathcal{K}_g} \beta_g \theta ]$, which
consequently vanishes.  By (\ref{zedP2}) we conclude that
\begin{align}
\label{fteqn3}
\nabla \tilde{\mathcal{F}}_{\tau} (\tilde{g}) = 0
\end{align}
as claimed.

 Next, taking a trace of \eqref{fteqn3}, yields an equation of the form
\begin{align}
\label{reqn}
\Delta_{\tilde{g}} R_{\tilde{g}} = Rm_{\tilde{g} }* Rm_{\tilde{g}} + \lambda \cdot \tilde{g}.
\end{align}
Since $\tilde{g}$ is assumed to be only $C^{4,\alpha}$, around any point $p \in M$,
there is a coordinate system $\{x^i\}$ so that the component functions $\tilde{g}_{ij}$ are
$C^{4,\alpha}$. Equation \eqref{reqn} is easily seen to be an elliptic
equation, and by elliptic regularity we conclude that $R_{\tilde{g}} \in C^{4,\alpha}$.
The equation \eqref{fteqn3} then implies that $\Delta_{\tilde{g}} Ric_{\tilde{g}} \in
C^{2, \alpha}$, which implies that $Ric_{\tilde{g}} \in C^{4,\alpha}$.
Since $\tilde{g} \in C^{4, \alpha}$, there exists a harmonic coordinate
system $\{y^i\}$ around $p$ such that the equation
\begin{align}
\frac{1}{2}\tilde{g}^{ij}\partial^{2}_{ij} \tilde{g}_{kl}+Q_{kl}(\partial \tilde{g}, \tilde{g})
=-Ric_{kl}(\tilde{g}),\label{eqsrt12}
\end{align}
where $Q(\partial \tilde{g},\tilde{g})$ is an expression that is quadratic in
$\partial \tilde{g}$,
polynomial in $\tilde{g}$ and has  $\sqrt{|\tilde{g}|}$ in its denominator
\cite{Petersen}. From this we conclude that $\tilde{g}_{ij} \in C^{5, \alpha}$.
A bootstrap argument shows that $\tilde{g}_{ij} \in C^{\ell, \alpha}$
for any $\ell > 0$.


Part $(iii)$ follows from a result, known as ``Ebin-Palais slicing''
\cite{Ebin, Palais}.  Our treatment is based on an argument of
Cheeger-Tian \cite[Theorem 3.1]{CheegerTian1}, see also \cite{FischerMarsden}.

\begin{lemma} \label{BasicSlice}  For each metric $g_1$ in a sufficiently small $C^{\ell+1,\alpha}$-neighborhood of $g$ ($\ell \geq 1$),
there is a $C^{\ell+2,\alpha}$-diffeomorphism $\phi : M \rightarrow M$ and a constant $c$ such that
\begin{align} \label{thtildef1}
\tilde{\theta} \equiv e^c \phi^{*}g_1 - g
\end{align}
satisfies
\begin{align} \label{divcon1}
\beta_g \tilde{\theta} = 0,
\end{align}
where $\beta_g$ is defined in \eqref{betadef}, and
\begin{align}
\label{trcon1}
\int tr_g \tilde{\theta}\ dV_g = 0.
\end{align}
Moreover, we have the estimate
\begin{align} \label{tilsmall1}
\| \tilde{\theta} \|_{C^{\ell+1,\alpha}} \leq C \| \theta_1 \|_{C^{\ell+1,\alpha}},
\end{align}
where $C = C(g)$.
\end{lemma}

\begin{proof} Let $\{ \omega_1, \dots , \omega_{\kappa} \}$ denote a basis of the space of conformal Killing forms with respect to $g$.
Consider the map
\begin{align}
\mathcal{N} : C^{\ell +2,\alpha}(TM) \times \mathbb{R} \times \mathbb{R}^{\kappa}  \times C^{\ell +1,\alpha}(S^2(T^{*}M)) \rightarrow C^{\ell,\alpha}(T^{*}M) \times \mathbb{R}
\end{align}
 given by
\begin{align} \label{Ndef} \begin{split}
&\mathcal{N}(X,r,v,\theta) = \mathcal{N}_{\theta}(X,r,v) \\
&= \big( \beta_g \big[ \phi_{X,1}^{*}(g + \theta) \big]  + \sum_i v_i \omega_i, e^r \int tr_g  \big[ \phi_{X,1}^{*}(g + \theta) \big]\ dV_g - n Vol(g) \big),
\end{split}
\end{align}
where $\phi_{X,1}$ denotes the diffeomorphism obtained by following the flow
generated by the vector field $X$ for unit time.
It is clear that $\mathcal{N}$ is a continuously differentiable mapping,
and linearizing in $(X,r,v)$ at $(X,r,v,\theta) = (0,0,0,0)$, we find
\begin{align*}
\mathcal{N}'_0 (Y,s,a) &= \frac{d}{d\epsilon} \big( \beta_g \big[ \phi_{\epsilon Y,1}^{*}(g) \big] + \sum_i (\epsilon a_i) \omega_i,e^{\epsilon s } \int tr_g  \big[ \phi_{\epsilon Y,1}^{*}(g) \big]\ dV_g - n Vol(g) \big) \Big|_{\epsilon = 0} \\
&= \big( \beta_g [\mathcal{L}_g Y^{\flat}] + \sum_i a_i \omega_i, n Vol(g) \cdot s \big) \\
&= \big( \Box_{\mathcal{K}_g} Y^{\flat} + \sum_i a_i \omega_i, n Vol(g) \cdot s \big),
\end{align*}
where $Y^{\flat}$ is the dual one-form to $Y$.

The adjoint map $(\mathcal{N}'_0)^{*} : C^{m + 2,\alpha}(T^{*}M) \times \mathbb{R} \rightarrow C^{m,\alpha}(TM) \times \mathbb{R} \times \mathbb{R}^{\kappa}$ is given by
\begin{align*}
(\mathcal{N}'_0)^{*}(\eta, s) = \big( (\Box_{\mathcal{K}_g} \eta)^{\sharp}, n Vol(g) s, \int \langle \eta, \omega_i \rangle\ dV \big),
\end{align*}
where $(\Box_{\mathcal{K}_g} \eta)^{\sharp}$ is the vector field dual to $\Box_{\mathcal{K}_g} \eta$. Setting $(\mathcal{N}'_0)^{*}(\eta, s) = 0$, we get $s=0$ and
\begin{align*}
\Box_{\mathcal{K}_g} \eta &= 0,  \\
\int \langle \eta, \omega_i \rangle\ dV &= 0, \ 1 \leq i \leq \kappa.
\end{align*}
The first equation implies that $\eta$ is a conformal Killing field, while the second implies that $\eta$ is orthogonal (in $L^2$) to the space
of conformal Killing forms.  It follows that $\eta = 0$, so the map $\mathcal{N}'_0$ is surjective.  By
a standard implicit function theorem argument, given $\theta_1 \in C^{\ell+1,\alpha}(S^2(T^{*}M))$ small enough we can solve the equation $\mathcal{N}_{\theta_1} = (0,0)$; i.e.,
there is a vector field $X \in C^{\ell+2,\alpha}(TM)$, a $v \in \mathbb{R}^{\kappa}$, and a real number $r$ such that
\begin{align} \label{zedN01}
\beta_{g} [ \phi^{*}g_1] + \sum_i v_i \omega_i = 0, \quad e^r \int tr_g  \big[ \phi_{X,1}^{*}(g_1) \big]\ dV_g = n Vol(g),
\end{align}
where $\phi = \phi_{X,1}$.  In addition, if we let
\begin{align*}
 \tilde{\theta} = e^r \phi^{*}g_1 - g,
\end{align*}
then $\tilde{\theta}$ satisfies (\ref{tilsmall1}).

For the first equation in (\ref{zedN01}), fix $1 \leq \ell \leq \kappa$ and pair both sides (in $L^2$) with $\omega_{\ell}$; then $v_{\ell} = 0$ since $\beta^{*}_g \omega_{\ell} = 0$.  Therefore, $\beta_{g} [ \phi^{*}g_1] = 0$.
Also,
\begin{align*}
\beta_g \tilde{\theta} = \beta_g[ e^r \phi^{*}g_1 - g] = 0,
\end{align*}
which proves (\ref{divcon1}). The trace condition (\ref{trcon1}) follows from the second condition in (\ref{zedN01}):
\begin{align*}
\int tr_g \tilde{\theta}\ dV_g &= \int tr_g \big[ e^r \phi^{*}g_1 - g \big]\ dV_g
=  \int \big\{ e^r \cdot tr_g [ \phi^{*} g_1 ] - n \big\}\ dV_g \\
&= e^r \int tr_g [ \phi^{*}g_1] \ dV_g - n Vol(g) = 0.
\end{align*}

\end{proof}

To prove the theorem, suppose $g_1$ is a critical metric which is $C^{k+1,\alpha}$-close to $g$, for some $k \geq 3$:
\begin{align} \label{g1eqn}
\nabla \tilde{\mathcal{F}}_{\tau} (g_1) = 0.
\end{align}
Applying the lemma, we obtain a $C^{k+2,\alpha}$-diffeomorphism $\phi : M \rightarrow M$ and a constant $c$ such that
$\tilde{\theta} = e^c \phi^{*}g_1 - g$ satisfies (\ref{trcon}).  Also, since $\tilde{\theta}$ satisfies (\ref{divcon1}),
\begin{align*}
P_g(\tilde{\theta}) &= \nabla \tilde{\mathcal{F}}_{\tau} (g + \tilde{\theta}) + \frac{1}{2} \mathcal{K}_{g + \tilde{\theta}} [ \Box_{\mathcal{K}_g} \beta_g \tilde{\theta}]
= \nabla \tilde{\mathcal{F}}_{\tau} (\phi^{*}g_1)
= \phi^{*}\{ \nabla \tilde{\mathcal{F}}_{\tau} (g_1)   \}
= 0,
\end{align*}
verifying (\ref{Pcon}).

To verify ellipticity, let $S_g : S^2(T^{*}M) \rightarrow S^2(T^{*}M)$ denote the linearization of $P_g$ at $\theta = 0$:
\begin{align} \label{Sdef}
S_g h = \frac{d}{ds} P_g (g + sh) \big|_{s=0}
= (\nabla \tilde{\mathcal{F}}_{\tau})_g'(h)
+ \frac{1}{2} \mathcal{K}_g [ \Box_{\mathcal{K}_g} \beta_g h].
\end{align}
Note that
\begin{align*}
\gamma'(0) = \frac{d}{ds} \gamma(sh) \big|_{s=0} = 0
\end{align*}
at a critical metric, since linearizing the definition (and using $\nabla \tmf_{\tau}[g] = 0$) gives
\begin{align*}
\gamma'(0) = \frac{1}{n}Vol(g)^{-1} \int tr_g [(\tmf_{\tau})'h]\ dV,
\end{align*}
which can easily be seen to vanish by linearizing (\ref{trGr}).

By the formulas of Section \ref{LocalSec},
\begin{align} \label{prinF}
(\nabla \tilde{\mathcal{F}}_{\tau})_{ij} = - \Delta R_{ij} + (1 + 2\tau) \nabla_i \nabla_j R - (\frac{1}{2} + 2\tau)( \Delta R) g_{ij} + \cdots,
\end{align}
where ``$\cdots$'' denotes terms which involve at most two derivatives of the metric.  Linearizing this expression,
\begin{align*}
(\nabla \tilde{\mathcal{F}}_{\tau})'h_{ij} = - \Delta R_{ij}' + (1 + 2\tau) \nabla_i \nabla_j R' - (\frac{1}{2} + 2\tau)( \Delta R') g_{ij} + \cdots,
\end{align*}
where ``$\cdots$'' now indicates terms which involve at most two derivatives of $h$.  By standard formulas for the linearization of the
Ricci and scalar curvatures,
\begin{align*}
R_{ij}' &= -\frac{1}{2}\big[ \Delta h_{ij} - \nabla_i \delta_j h - \nabla_j \delta_j h + \nabla_i \nabla_j (tr\ h) \big] + \cdots, \\
R' &= -\Delta( tr\ h) + \delta^2 h + \cdots.
\end{align*}
Consequently,
\begin{align} \label{linF}
\begin{split}
(\nabla \tilde{\mathcal{F}}_{\tau})_g'h_{ij} &= \frac{1}{2} \Delta^2 h_{ij} - \frac{1}{2} \Delta \big[ \nabla_i \delta_j h + \nabla_j \delta_i h \big]  - (2\tau + \frac{1}{2}) \nabla_i \nabla_j (\Delta tr\ h) \\
 & \hskip.25in + (2\tau+1) \nabla_i \nabla_j (\delta^2 h) + (2\tau + \frac{1}{2}) \big[\Delta^2 (tr\ h) -\Delta (\delta^2 h)\big] g_{ij} + \cdots.
\end{split}
\end{align}
Also, a simple calculation gives
\begin{align*}
\frac{1}{2} \mathcal{K}_g [ \Box_{\mathcal{K}_g} \beta_g h]_{ij} &= \frac{1}{2} \Delta \big[ \nabla_i \delta_j h + \nabla_j \delta_i h \big]
 - \frac{2(n-1)}{n^2}\nabla_i \nabla_j (\Delta tr\ h) + \frac{n-2}{n} \nabla_i \nabla_j (\delta^2 h) \\
& \hskip.5in  + \frac{2(n-1)}{n^3} \Delta^2 (tr\ h) g_{ij} - \frac{2(n-1)}{n^2}\Delta( \delta^2 h) g_{ij} + \cdots.
\end{align*}
Therefore,
\begin{align} \label{prinS}
\begin{split}
Sh_{ij} &= \frac{1}{2} \Delta^2 h_{ij}  - \Big[ 2\tau + \frac{n^2 + 4n - 4}{2n^2}\Big] \nabla_i \nabla_j (\Delta tr\ h) + 2\big[ \tau+ \frac{n-1}{n}\big] \nabla_i \nabla_j (\delta^2 h) \\
 & \hskip.25in  + \Big[ 2\tau + \frac{n^3 + 4n - 4}{2n^3} \Big]\Delta^2 (tr\ h)g_{ij} - \Big[ 2\tau + \frac{n^2 + 4n - 4}{2n^2}\Big] \Delta (\delta^2 h) g_{ij} + \cdots
\end{split}
\end{align}
This implies the symbol $(\sigma_{\xi}S) h$ is given by
\begin{align} \label{symS}
\begin{split}
(\sigma_{\xi}S) h_{ij} &= \frac{1}{2} |\xi|^4 h_{ij} - \Big[ 2\tau + \frac{n^2 + 4n - 4}{2n^2}\Big] \xi_i \xi_j |\xi|^2 (tr\ h) + 2\big[ \tau+ \frac{n-1}{n}\big] \xi_i \xi_j h_{k\ell} \xi_k \xi_{\ell} \\
& \hskip.25in  + \Big[ 2\tau + \frac{n^3 + 4n - 4}{2n^3} \Big]  |\xi|^4 (tr\ h) g_{ij} - \Big[ 2\tau + \frac{n^2 + 4n - 4}{2n^2}\Big]  |\xi|^2 h_{k\ell} \xi_k \xi_{\ell} g_{ij}.
 \end{split}
 \end{align}
To see that the symbol is non-degenerate, suppose $(\sigma_{\xi}S) h = 0$ and $\xi \neq 0$.  Then
\begin{align} \label{symzed}
0 = \frac{1}{|\xi|^2} (\sigma_{\xi}S)h_{ij} \xi_i \xi_j = \frac{2(n^2-1)}{n^2}\Big[ h_{k\ell} \xi_k \xi_{\ell} - \frac{1}{n}|\xi|^2(tr\ h) \Big]|\xi|^2,
\end{align}
hence
\begin{align} \label{trsig}
h_{k\ell} \xi_k \xi_{\ell} = \frac{1}{n}|\xi|^2 (tr\ h).
\end{align}
Also, taking the trace of (\ref{symS}),
\begin{align*}
0 = g^{ij} (\sigma_{\xi}S)h_{ij}  = &\big[ 2\tau(n-1) + \frac{n}{2} \big] |\xi|^4 (tr\ h) \\
&- \big[ 2\tau(n-1) + \frac{n}{2}\big]|\xi|^2 h_{k\ell}\xi_k \xi_{\ell}.
\end{align*}
In view of (\ref{trsig}), this implies
\begin{align*}
0 = \frac{2(n-1)^2}{n} \big[ \tau + \frac{n}{4(n-1)} \big] |\xi|^4 (tr\ h).
\end{align*}
Since $\tau \neq -n/4(n-1)$, we conclude that $tr\ h = 0$.  As a consequence of this and (\ref{symzed}), the symbol reduces to
\begin{align*}
0 = (\sigma_{\xi}S)h_{ij} = \frac{1}{2} |\xi|^4 h_{ij},
\end{align*}
so that $h_{ij} = 0$.  It follows that the symbol is injective and $S$ is elliptic.
\end{proof}


\subsection{The Bach tensor}
\label{BachSec}
The Bach tensor in dimension four was introduced in Section~\ref{Bachsub}.
In this section, we study the local properties of the moduli
space of Bach-flat metrics, in analogy with Section \ref{Moduli} above.
Suppose $g$ is a Bach-flat metric, and let $S^2_0(T^{*}M)$ denote the
space of trace-free (with respect to $g$) symmetric $(0,2)$-tensor fields.
Consider
\begin{align} \label{PBdef}
P^{B}_g \theta = B(g+\theta) + \mathcal{K}_{g + \theta}(\Box_{\mathcal{K}_g} \beta_g \theta) - \frac{1}{4} \tr \big\{ B(g+\theta) + \mathcal{K}_{g + \theta}(\Box_{\mathcal{K}_g} \beta_g \theta)\big\}g.
\end{align}
The $g$-term ensures that $P^{B}_g : S^2_0(T^{*}M) \rightarrow S^2_0(T^{*}M)$.

\begin{theorem}  \label{PBzed}  Assume $(M,g)$ is Bach-flat,
and consider $P_g^B$ as a mapping from $C^{k+1, \alpha}$ to $C^{k-3,\alpha}$ for $k \geq 3$.
\vspace{1mm}

\noindent $(i)$ The linearization of $P^B_g$ at $\theta = 0$ is an elliptic operator.
\vspace{1mm}

\noindent $(ii)$ If $\theta$ is sufficiently small and $P^{B}_g(\theta) = 0$,
then for each $p \in M$, there exists a locally defined function
$u : U \rightarrow \RR$, where $U$ is a neighborhood of $p$, such
that the metric $e^u ( g + \theta)$ is a smooth Riemannian metric and is Bach-flat.
\vspace{1mm}

\noindent $(iii)$  Conversely, if $g_1 = g + \theta_1$ is a Bach-flat metric in a sufficiently small $C^{k+1,\alpha}$-neighborhood of $g$ ($k \geq 3$), then there exists a $C^{k+2,\alpha}$-diffeomorphism $\phi : M \rightarrow M$ and a function $u \in C^{k+1,\alpha}(M)$
such that
\begin{align} \label{thtildef2}
e^u \phi^{*}g_1 = g + \tilde{\theta}
\end{align}
with $\tilde{\theta} \in S^2_0(T^{*}M)$ and
\begin{align} \label{PBcon}
P^B_g(\tilde{\theta}) = 0.
\end{align}
\end{theorem}

\begin{proof}  We will denote the linearized operator by $S_g^B$.
The proof of ellipticity of $S_g^B$
is essentially the same as the proof in Theorem \ref{Pzed}, since the
Bach tensor is a multiple of the Euler-Lagrange equations of $\mathcal{F}_{-1/3}$.

For part $(ii)$, assume $P^B_g (\theta) = 0$, and let $\tilde{g} = g + \theta$.  Then
\begin{align} \label{Pzero}
0 = B(\tilde{g}) + \ck_{\tilde{g}}(\Box_{\mathcal{K}_g} \beta_g \theta) - \frac{1}{4} \tr \big\{ B(\tilde{g}) + \ck_{\tilde{g}}(\Box_{\mathcal{K}_g} \beta_g \theta)\big\}g.
\end{align}
Taking the trace with respect to $\tilde{g}$, we get
\begin{align} \label{tiltrace}
 0 = - \frac{1}{4} \mbox{tr}_{\tilde{g}} \big\{ B(\tilde{g}) + \ck_{\tilde{g}}(\Box_{\mathcal{K}_g} \beta_g \theta)\big\}( \mbox{tr}_{\tilde{g}}\ g).
\end{align}
Since $\tilde{g}$ is a small perturbation of $g$, the trace of $g$ with respect to $\tilde{g}$ is
non-vanishing, and it follows that
\begin{align*}
\mbox{tr}_{\tilde{g}} \big\{ B(\tilde{g}) + \ck_{\tilde{g}}(\Box_{\mathcal{K}_g} \beta_g \theta)\big\} = 0.
\end{align*}
Substituting this into (\ref{Pzero}) gives
\begin{align} \label{BnK}
0 = B(\tilde{g}) + \mathcal{K}_{\tilde{g}}(\Box_{\mathcal{K}_g} \beta_g \theta).
\end{align}
Pairing both sides of (\ref{BnK}) with $B(\tilde{g})$ and integrating gives
\begin{align} \label{Bzero}
 0 = \int |B(\tilde{g})|_{\tilde{g}}^2\ dV_{\tilde{g}} + \int \langle B(\tilde{g}), \ck_{\tilde{g}}(\Box_{\mathcal{K}_g} \beta_g \theta) \rangle\ dV_{\tilde{g}}.
\end{align}
(Note that, as in the proof of Theorem \ref{Pzed}, we should first mollify $\theta$, then use a limiting argument to show that the Bach tensor vanishes, but as the argument
is identical we will omit the details.) Since
\begin{align*}
\big( \mathcal{K} \big|_{S^2_0(T^{*}M)} \big)^{*} = -2 \delta,
\end{align*}
it follows that
\begin{align*}
\int \langle B(\tilde{g}),  \ck_{\tilde{g}}(\Box_{\mathcal{K}_g} \beta_g \theta) \rangle\ dV_{\tilde{g}} &= \int \langle (\ck_{\tilde{g}})^{\ast}B(\tilde{g}), \ck_{\tilde{g}}(\Box_{\mathcal{K}_g} \beta_g \theta) \rangle\
dV_{\tilde{g}} \\
&= \int \langle -2 \delta_{\tilde{g}} B(\tilde{g}), \ck_{\tilde{g}}(\Box_{\mathcal{K}_g} \beta_g \theta) \rangle\ dV_{\tilde{g}} = 0,
\end{align*}
since the Bach tensor is divergence-free.
It follows from (\ref{Bzero}) that $B(\tilde{g})= 0$.

 To prove the local regularity in part $(ii)$, one can locally solve
the negative Yamabe problem with Dirichlet boundary data, and the
Bach-flat equation for the conformal metric becomes
\begin{align}
\Delta Ric = Rc * Rm,
\end{align}
and higher regularity follows using harmonic coordinates and
a bootstrap argument as in part $(ii)$ of Theorem~\ref{Pzed};
this was also proved in \cite[Theorem~6.4]{TV2}.


To prove part $(iii)$, we use a slight variant of the method used to prove part $(iii)$ of Theorem \ref{Pzed}.  Suppose $g_1$ is a Bach-flat metric
 in a $C^{k+1,\alpha}$-neighborhood of $g$:
\begin{align} \label{g2eqn}
B (g_1) = 0.
\end{align}
Let $\{ \omega_1, \dots , \omega_{\kappa} \}$ denote a basis of the space of conformal Killing forms with respect to $g$, and
consider the map
\begin{align}
\mathcal{H} : C^{k+2,\alpha}(TM) \times C^{k+1,\alpha}(M) \times \mathbb{R}^{\kappa}  \times C^{k+1,\alpha}(S^2(T^{*}M)) \rightarrow C^{k,\alpha}(T^{*}M) \times C^{k+1,\alpha}(M)
\end{align}
given by
\begin{align} \label{Ndef2} \begin{split}
\mathcal{H}(X,u,v,\theta) &= \mathcal{H}_{\theta}(X,u,v) \\
 &= \big( \beta_g \big[ e^u \phi_{X,1}^{*}(g + \theta) \big] + \sum_i v_i \omega_i ,e^u  tr_g  \big[ \phi_{X,1}^{*}(g + \theta) \big]  - 4  \big),
\end{split}
\end{align}
where $\phi_{X,1}$ denotes the diffeomorphism obtained by following the flow generated by the vector field $X$ for unit time.

Linearizing $\mathcal{H}$ in $(X,u,v)$ at $(X,u,v,\theta) = (0,0,0,0)$, we find
\begin{align*}
\mathcal{H}'_0 (Y,u,a) &= \frac{d}{d\epsilon} \big( \beta_g \big[ e^{\epsilon u} \phi_{\epsilon Y,1}^{*}(g) \big] + \sum_i (\epsilon a_i) \omega_i ,e^{\epsilon u }   tr_g  \big[ \phi_{\epsilon Y,1}^{*}(g) \big] - 4   \big) \Big|_{\epsilon = 0} \\
&= \big( \beta_g [ u g] + \beta_g \mathcal{L}_g Y^{\flat} + \sum_i a_i \omega_i, 4  u + tr_g \mathcal{L}_g Y^{\flat} \big) \\
&= \big( \Box_{\mathcal{K}_g} Y^{\flat} + \sum_i a_i \omega_i, 4 u + 2 \delta_g Y^{\flat} \big),
\end{align*}
where $Y^{\flat}$ is the dual one-form to $Y$.  The adjoint map $(\mathcal{H}'_0)^{*} : C^{m+2,\alpha}(T^{*}M) \times C^{m+1,\alpha}(M) \rightarrow C^{m,\alpha}(TM) \times \mathbb{R}^{\kappa} \times C^{m+1,\alpha}(M)$ is given by
\begin{align*}
(\mathcal{H}'_0)^{*}(\eta, f) = \big( (\Box_{\mathcal{K}_g} \eta)^{\sharp} - 2 \nabla f,  4 f, \int \langle \eta, \omega_i \rangle\ dV \big),
\end{align*}
where $(\Box_{\mathcal{L}_g} \omega)^{\sharp}$ is the vector field dual to $\Box_{\mathcal{L}_g} \eta$.
It follows that if $(\eta, f) \in \mbox{ker } (\mathcal{H}'_0)^{*}$, then $f = 0$ and
\begin{align*}
\Box_{\mathcal{K}_g} \eta &= 0, \\
\int \langle \eta, \omega_i \rangle\ dV &= 0, \ 1 \leq i \leq \kappa.
\end{align*}
Therefore, $\eta = 0$ and the map $\mathcal{H}'_0$ is surjective.
Consequently, by the implicit function theorem, given $\theta_1 \in C^{k+1,\alpha}(S^2(T^{*}M))$ small enough we can solve the equation
\begin{align*}
\mathcal{H}_{\theta_1}(X,u, v) = (0,0),
\end{align*}
i.e., there is a vector field $X \in C^{k+2,\alpha}(TM)$, a function $u \in C^{k+1,\alpha}(M)$,
and a vector $v \in \RR^{\kappa}$ such that
\begin{align} \label{zedN1} \begin{split}
\beta_g \big[ e^u \phi_{X,1}^{*}(g + \theta) \big] + \sum_i v_i \omega_i &= 0, \\
e^u  tr_g  \big[ \phi_{X,1}^{*}(g + \theta) \big]  - 4 &= 0.
\end{split}
\end{align}
As in the proof of Lemma \ref{BasicSlice}, the first equation implies
\begin{align*}
\beta_{g} [ e^u \phi^{*}g_1] = 0,
\end{align*}
where $\phi = \phi_{X,1}$.  If we write
\begin{align*}
 e^u \phi^{*}g_1 = g + \tilde{\theta},
\end{align*}
then the second equation in (\ref{zedN1}) implies
\begin{align*}
tr_g \tilde{\theta} = tr_g\ \big[e^u \phi^{*}g_1 - g \big] = tr_g  \big[ e^u \phi_{X,1}^{*}(g_1) \big]  - 4 = 0,
\end{align*}
hence $\tilde{\theta} \in S^2_0(T^{*}M)$.

It follows that the gauge-fixing term in (\ref{PBdef}) vanishes:
\begin{align*}
\mathcal{K}_{g + \tilde{\theta}}(\Box_{\mathcal{K}_g} \beta_g \tilde{\theta}) = 0.
\end{align*}
Therefore,
\begin{align*}
P^B_g(\tilde{\theta}) &=  B(g+ \tilde{\theta}) + \mathcal{K}_{g + \tilde{\theta}}(\Box_{\mathcal{K}_g} \beta_g \tilde{\theta}) - \frac{1}{4} \tr \big\{ B(g+\tilde{\theta}) + \mathcal{K}_{g + \tilde{\theta}}(\Box_{\mathcal{K}_g} \beta_g \tilde{\theta})\big\}g \\
&= B(g+ \tilde{\theta}) - \frac{1}{4} tr_g \{ B(g+\tilde{\theta}) \}g
= B(e^u \phi^{*}g_1) - \frac{1}{4} tr_g \{ B( e^u \phi^{*}g_1) \}g \\
&= e^{-u} B(\phi^{*}g_1) - \frac{1}{4} tr_g \{e^{-u} B( \phi^{*}g_1) \}g  \\
&= e^{-u} \phi^{*}(B(g_1)) - \frac{1}{4} tr_g \{ e^{-u} \phi^{*} (B(g_1)) \}g = 0,
\end{align*}
verifying (\ref{PBcon}).
\end{proof}

\subsection{The Kuranishi Map}  \label{KM}

In view of our preceding results, it is possible to describe the moduli space of critical metrics
near $g$ as the zeroes of a Fredholm map.  In fact,
using the Implicit Function Theorem the moduli space can be described as the zero
set of a finite-dimensional map called the {\em Kuranishi map}.  To make this more precise, we define $H^1_{\tau}$ to be
the kernel of the linearization of $P_g$:
\begin{align} \label{H1def}
H^1_{\tau} = H^1_{\tau}(M,g) = \big\{
h \in C^{4, \alpha} (\overline{S}_{0}^2(T^{*}M)) \ \big|\ S_g h = 0 \big\}.
\end{align}
In the case $n = 4$ and $\tau = -1/3$ (the Bach tensor), we restrict to
traceless tensors:
\begin{align} \label{BachH1def}
H^1_{-1/3} = H^1_{\tau}(M,g) = \big\{
h \in C^{4,\alpha}(S_{0}^2(T^{*}M))\ \big|\ S_g^B h = 0 \big\}.
\end{align}
Since $S_g$ is elliptic, $H^1_{\tau}$ is finite dimensional.
Moreover, we can characterize $H^1_{\tau}$ in
terms of the linearization of $\nabla \tmf_{\tau}$:
\begin{lemma} \label{H1Lemma}
We have
\begin{align*}
H^1_{\tau} = \big\{ h \in C^{\infty}(\overline{S}_{0}^2(T^{*}M))\ \big|\ (\nabla \tmf_{\tau})_g'h = 0,\ \beta_g h = 0 \big\}.
\end{align*}
In the case  $n = 4$ and $\tau = -1/3$,
\begin{align}
H^1_{-1/3} = \big\{ h \in C^{\infty}(S_{0}^2(T^{*}M))\ \big|\
B_g'h = 0,\ \delta_g h = 0 \big\}.
\end{align}
\end{lemma}
\begin{proof} Smoothness follows from ellipticity of $S_g$.
Recall from (\ref{Sdef}) that
\begin{align*}
S_g h = \frac{d}{ds} P_g (g + sh) \big|_{s=0}
= (\nabla \tilde{\mathcal{F}}_{\tau})_g'(h)
+ \frac{1}{2} \mathcal{K}_g [ \Box_{\mathcal{K}_g} \beta_g h].
\end{align*}
Clearly, it suffices to show that if $S_g h = 0$, then $\beta_g h = 0$.  To see this, we first note that given a
path of metrics $g(s)$, $-\epsilon < s < \epsilon$, with $g(0) = g$, then
\begin{align*}
\delta_{g(s)} [(\nabla \tmf_{\tau})_{g(s)}] = 0,
\end{align*}
since the gradient of a Riemannian functional is divergence-free.  If $g'(0) = h$, then linearizing this
identity at $s = 0$ gives
\begin{align*}
\delta_g [(\nabla \tmf_{\tau})_g'h] = 0.
\end{align*}
Therefore, if $S_g h = 0$, it follows that
\begin{align*}
0 &= \delta_g S_g h
= \delta_g [(\nabla \tmf_{\tau})_g'h]
+ \frac{1}{2} \delta_g \mathcal{K}_g [ \Box_{\mathcal{K}_g} \beta_g h]
= \frac{1}{2}\Box_{\mathcal{K}_g}^2 \beta_g h.
\end{align*}
Consequently, $\mathcal{K}_g \beta_g h = 0$.  But
\begin{align*}
\| \beta_g h \|_{L^2}^2 &= \big\langle \beta_g h , \beta_g h \big\rangle_{L^2}
= \big\langle h , \beta_g^{*} \beta_g h \big\rangle_{L^2}
= \big\langle  h , -\frac{1}{2} \mathcal{K}_g \beta_g h \big\rangle_{L^2}
= 0.
\end{align*}
The proof for the Bach tensor is identical.
\end{proof}

We will now show there is a well defined mapping $\Psi : H^1_{\tau} \rightarrow H^2_{\tau}$,
called the {\em{Kuranishi map}}, for which the
moduli space is locally given by $\Psi^{-1}(0)$,
where $H^2_{\tau}$ is the cokernel of the linearized operator. (Note that since $S_g$
is self-adjoint, on a compact manifold we have $H^2_{\tau} = H^1_{\tau}$.)
To do this, we need to prove a technical lemma about the
structure of the nonlinear terms in the Taylor expansion of $P = P_g$:
\begin{lemma}
\label{nstr}
There exists a constant $C$ such that if we write
\begin{align} \label{PS2}
P_g(h)= P_g(0) + S_gh + Q(h),
\end{align}
then for $h_1, h_2 \in C^{4,\alpha}$ of sufficiently small norm,
\begin{align} \label{quadstructure}
\Vert Q(h_1) -  Q(h_2)\Vert_{C^{\alpha}} \leq
C(  \Vert h_1 \Vert_{C^{4, \alpha}}
+ \Vert h_2 \Vert_{C^{4, \alpha}}) \cdot
\Vert  h_1 - h_2 \Vert_{C^{4, \alpha}}.
\end{align}
\end{lemma}
\begin{proof}
Since the proof involves a rather lengthy calculation we begin with a brief overview.  The tensor $\nabla \tmf_{\tau}$ can be schematically expressed
as
\begin{align} \label{Btscheme}
\nabla \tmf_{\tau}(g) = g * g^{-1} * g^{-1} * \nabla_g^2 Rm_g + g * g^{-1} * g^{-1} * Rm_g * Rm_g,
\end{align}
where $Rm_g$ denotes the curvature tensor of $g$, $g^{-1} * \cdots * g^{-1} * A * B$ denotes any
linear combination of terms involving contractions of the tensor product $A \otimes B$, and $g^{-1} * \cdots * g^{-1} * \nabla_g^k * A$ denotes linear combinations of contractions of the $k$-th iterated covariant derivative of $A$.
Since the mapping $P$ is defined by
\begin{align} \label{Pdef2}
P_g(h) = \nabla \tilde{\mathcal{F}}_{\tau} (g + h) + \frac{1}{2} \mathcal{K}_{g + h}[\Box_{\mathcal{K}_g} \beta_g  \theta]  - \gamma(h)\cdot g,
\end{align}
the first step in proving the estimates is to analyze the curvature term
\begin{align} \label{leadexp} \begin{split}
\nabla \tilde{\mathcal{F}}_{\tau} (g + h) &= (g+h) * (g+h)^{-1} * (g+h)^{-1} * \nabla_{g+h}^2 Rm_{g+h} \\
& \hskip.25in  + (g+h)*(g+h)^{-1} * (g + h)^{-1} * Rm_{g+h} * Rm_{g+h}.
\end{split}
\end{align}
Proceeding as in \cite[Proposition 3.1]{AV12}, the starting point is the formula
\begin{align} \label{CSdiff}
\Gamma(g + h)^{k}_{ij} = \Gamma(g)^{k}_{ij} + \frac{1}{2} (g+h)^{km}
\left\{ \nabla_j h_{im} + \nabla_i h_{jm} - \nabla_m h_{ij} \right\},
\end{align}
where $\Gamma(\cdot)$ denotes the Christoffel symbols of a metric.
Using this formula and the notation introduced above, we can express the covariant derivative with respect to the metric $g+h$ as
\begin{align} \label{DTId}
\nabla_{g+ h} T = \nabla_g T + (g +h)^{-1} * \nabla_g h * T,
\end{align}
where $T$ is any tensor field. Also, by the standard formula for the $(1,3)$-curvature tensor in terms of the Christoffel symbols we have
\begin{align}
\label{e1rm}
Rm_{g + h} = Rm_g + (g +h)^{-1} * \nabla^2 h +
(g + h)^{-2} * \nabla h * \nabla h.
\end{align}
In the following, any covariant derivative without a subscript
will mean with respect to the fixed metric $g$.

Taking two covariant derivatives $\nabla_{g+h}$ of $Rm_{g+h}$ and repeatedly using (\ref{DTId}), we obtain
\begin{align} \label{Btscheme2}
\begin{split}
&(g+h) *\left\{ (g+h)^{-2} * \nabla^2_{g+h} Rm_{g+h} + (g+h)^{-2} *  Rm_{g+h} * Rm_{g+h}\right\}  \\
&= (g+h)*\Big\{ (g+h)^{-2} * \nabla^2 Rm_g + (g+h)^{-2} * Rm_g * Rm_g  \\
&\ \ \ + (g + h)^{-3} * \nabla Rm_g *\nabla h +(g+h)^{-3} * Rm_g * \nabla^2 h \\
&\ \ \ + (g+h)^{-4} * Rm_g * \nabla h * \nabla h + (g+h)^{-3} * \nabla^4 h \\
&\ \ \ + (g+h)^{-4} * \nabla^3 h * \nabla h + (g+h)^{-4} * \nabla^2 h  * \nabla^2 h  \\
&\ \ \ + (g+h)^{-5} *  \nabla^2 h * \nabla h * \nabla h + (g + h)^{-6} *   \nabla h * \nabla h  * \nabla h * \nabla h \Big\}.
\end{split}
\end{align}
Taking into account the term $\gamma(h) g $, and the gauge-fixing
term in \eqref{Pdef2}, we obtain
\begin{align} \label{PS1} \begin{split}
P_g(h) & = (g+h)* \Big\{ (g+h)^{-2} * \nabla^2 Rm_g + (g+h)^{-2} * Rm_g * Rm_g  \\
&\ \ \ + (g + h)^{-3} * \nabla Rm_g *\nabla h + (g+h)^{-3} * Rm_g * \nabla^2 h \\
&\ \ \ + (g+h)^{-4} * Rm_g * \nabla h * \nabla h + (g+h)^{-3} * \nabla^4 h \\
&\ \ \ +  (g+h)^{-1} * g^{-3} * g * \nabla^4 h  +  (g+h)^{-4} * \nabla^3 h * \nabla h \\
&\ \ \ + (g+h)^{-2} *g*  g^{-3} * \nabla^3 h * \nabla h
+ (g+h)^{-4} * \nabla^2 h  * \nabla^2 h  \\
&\ \ \ + (g+h)^{-5} *  \nabla^2 h * \nabla h * \nabla h\\
&\ \ \ + (g + h)^{-6} *   \nabla h * \nabla h  * \nabla h * \nabla h  \Big\} + \int \big\{ \cdots \big\}\ dV,
\end{split}
\end{align}
where $\{ \cdots \}$ denotes the terms in (\ref{Btscheme2}).

Since we are trying to estimate the remainder terms in the Taylor expansion of $P(h)$, we want to write the above expression in terms of its linearization; i.e.,
\begin{align*}
P_g(h) &= P_g(0) + S_g h + \cdots \\
&= \nabla \tmf_{\tau}(g) + S_g h + \cdots
\end{align*}
To do this, we use the identity (which holds for $h$ small and all
integers $k \geq 1$)
\begin{align} \label{remk}
(g+h)^{-k} - g^{-k} = g^{-k - 1} * h + r_k(h),
\end{align}
where the remainder satisfies
\begin{align} \label{rkdiff}
| r_k(h_1) - r_k(h_2) | \leq C_k(g) \big( |h_1| + |h_2| \big) |h_1 - h_2|,
\end{align}
with a similar estimate for the H\"older norm.

Next, we substitute (\ref{remk}) into each term of (\ref{PS1}) involving a power of $(g+h)^{-1}$, then collect all terms which are zeroth order in $h$ (which combine to give $P_g(0)$),
those which are linear in $h$ (which combine to give $Sh$), and those which are higher order in $h$.

For example, consider the term $(g+h)*(g+h)^{-3} * \nabla^4 h$.
Applying (\ref{remk}), we may write this as
\begin{align} \begin{split}
\label{excase}
(g+h)*(g+h)^{-3} * \nabla^4 h
&= g * g^{-3} * \nabla^4 h + g * g^{-4} * h * \nabla^4 h + g * g^{-4}* r_3(h) * \nabla^4 h \\
& + g^{-3} * h * \nabla^4 h + g^{-4} * h * h * \nabla^4 h + g^{-4}* r_3(h) * \nabla^4 h.
\end{split}
\end{align}
It is clear that only the first term on the right hand side contributes
to the linearization, hence we need to estimate the rest of the terms.  Consider the next term on the right-hand side of
(\ref{excase}):
\begin{align} \label{Ttdef}
T(h) = g * g^{-4} * h * \nabla^4 h.
\end{align}
Then
\begin{align*}
T(h_1) - T(h_2) &= g * g^{-4} * h_1 * \nabla^4 h_1 - g * g^{-4} * h_2 * \nabla^4 h_2 \\
&= g * g^{-4} * (h_1 - h_2) * \nabla^4 h_1 + g * g^{-4} * h_2 * \nabla^4 (h_1 - h_2),
\end{align*}
hence
\begin{align*}
|T(h_1) - T(h_2)|  & \leq |h_1 - h_2| |\nabla^4 h_1|   + |h_2| | \nabla^4 (h_1 - h_2)|
\end{align*}
and it follows that
\begin{align} \label{semi} \begin{split}
\|T(h_1) - T(h_2)\|_{C^0} &\leq \Big\{ \|h_1 - h_2\|_{C^0}  \| h_1\|_{C^4}  + \|h_2\|_{C^0} \| h_1 - h_2\|_{C^4} \Big\} \\
& \leq C(  \Vert h_1 \Vert_{C^{4}}
+ \Vert h_2 \Vert_{C^{4}}) \cdot
\Vert  h_1 - h_2 \Vert_{C^{4}}.
\end{split}
\end{align}
Using \eqref{rkdiff}, the rest of the terms on the right hand side of \eqref{excase} can be
estimated similarly. Estimating difference quotients gives the required Holder estimate.
Analyzing each term in \eqref{PS1} in a similar fashion,
\eqref{quadstructure} follows..
\end{proof}

The main consequence of this formula is the following:
\begin{corollary}
\label{Kurcor}
The zero set of $P_g$ is locally modeled on the
zero set of a finite-dimensional map
\begin{align}
\Psi : H^1_{\tau} \rightarrow H^2_{\tau} \simeq H^1_{\tau}.
\end{align}
Consequently, if $g$ is infinitesimally
rigid, then $g$ is rigid.
In the case $n =4$ and $\tau = -1/3$, infinitesimal
Bach-rigidity implies Bach-rigidity.

In general, the map $\Psi$ is equivariant with respect to the
isometry group $\mathcal{I}(g)$.
If $\mathcal{C}(g) = \mathcal{I}(g)$,
where $\mathcal{C}(g)$ is the conformal automorphism group of $g$,
then the moduli space of solutions, up to
diffeomorphism and scaling, is given by the space of orbits
$\Psi^{-1}(0) / \mathcal{I}(g)$.
\end{corollary}
\begin{proof}
The existence of the map $\Psi$ follows from Lemma \ref{nstr},
together with a fixed-point argument as in \cite[Proposition 4.2.1]{RollinSinger}.
The argument above identified the zero set of $P_g$ with the
moduli space of solutions modulo diffeomorphism and scaling,
up to uniqueness of the gauge $\beta_g( \tilde{\theta}) = 0$ found
in Lemma \ref{BasicSlice}. This gauge is unique up
to the action of the conformal group, since the kernel of the
linearization $\mathcal{N}_0'$ is the space of conformal Killing
fields. Under the assumption that $\mathcal{C}(g) = \mathcal{I}(g)$, the
gauge is therefore unique up to the action of the isometry group.
The same argument holds for $P_g^B$ in the case $n=4, \tau = -1/3$.
\end{proof}
\begin{remark}{\em
It might seem more natural to use the gauge $\delta_g \tilde{\theta} = 0$ in
Lemma~\ref{BasicSlice}, instead of the gauge $\beta_g \tilde{\theta} = 0$.
With a suitable modification of the map $P_g$, this approach indeed works,
thus obviating the assumtion that
$\mathcal{C}(g) = \mathcal{I}(g)$ in the second part of Corollary~\ref{Kurcor}.
However, it is substantially more efficient to do compuations
in the $\beta_g \tilde{\theta} = 0$ gauge, so we prefer to work in this gauge.
There is no loss of generality for Einstein metrics,
since for these it is true that $\mathcal{C}(g) = \mathcal{I}(g)$, with the
only exceptional case being that of the round sphere.
}
\end{remark}

\section{Second variation at an Einstein metric}
\label{SV}

We begin this section by expressing the second variation
of the volume normalized functional in terms of the
un-normalized functional. In the latter part of the
section, we will compute the full second variation of the
un-normalized functional. These will then be combined
later in Sections \ref{TTsec} and \ref{confsec}.

\subsection{Normalized functional}
As before, we let $\mf_{\tau}$ denote any quadratic curvature
functional, and $\tmf_{\tau}$ denote the associated volume normalized
functional.

\begin{lemma}If $g$ is critical for $\tmf_{\tau}$ then
\begin{align}
\label{tdd}
(\tmf_{\tau}[g_s])''
= -p(p+1) (Vol)^{p-2} \big( (Vol)' \big)^2 \cdot \mf_{\tau}
+ p (Vol)^{p-1} (Vol)'' \cdot \mf_{\tau} + (Vol)^p \cdot (\mf_{\tau})''.
\end{align}
\end{lemma}

\begin{proof}
Differentiating \eqref{oned}, we obtain
\begin{align}
\begin{split}
(\tmf_{\tau}[g_s])'' &= p(p-1) (Vol)^{p-2} (Vol')(Vol') \cdot \mf_{\tau}
+ p (Vol)^{p-1} (Vol)'' \cdot \mf_{\tau}\\
& + 2 p (Vol)^{p-1} (Vol)' \cdot (\mf_{\tau}[g_s])'
+ (Vol)^p   (\mf_{\tau}[g_s])''.
\end{split}
\end{align}
The result follows by using \eqref{foc} on the third term on the
right hand side and combining with the first term to give
\begin{align}
\begin{split}
& p(p-1) (Vol)^{p-2} (Vol')(Vol') \cdot \mf_{\tau}
+  2 p (Vol)^{p-1} (Vol)' ( -p (Vol)^{-1} (Vol)' \cdot \mf_{\tau}) \\
& = - p (p+1)  (Vol)^{p-2} \big( (Vol') \big)^2 \cdot \mf_{\tau}.
\end{split}
\end{align}
\end{proof}

Next, we compute the second derivative of the unnormalized functional
$\mf_{\tau}$.

\begin{lemma}If $\nabla \mf_{\tau} = c \cdot g$, then
\begin{align}
\label{fdd}
\begin{split}
\frac{d^2}{ds^2} (\mf_{\tau}[g_s]) &= \int \langle (\nabla \mf_{\tau})' (g_s) g_s', g_s' \rangle dV\\
&- 2 c \int |g_s'|^2 dV + c \int (tr_{g_s} g_s'') dV + \frac{c}{2} \int  (tr_{g_s} g_s')^2 dV.
\end{split}
\end{align}
\end{lemma}

\begin{proof}
For any path of metrics $g_s$, we have
\begin{align}
\frac{d}{ds} (\mf_{\tau}[g_s]) = \int \langle \nabla \mf_{\tau} (g_s), g_s' \rangle dV,
\end{align}
where $\nabla \mf_{\tau}$ are the Euler-Lagrange equations. The integrand is locally
\begin{align}
g^{ip} g^{jq} ( \nabla \mf_{\tau} (g_s))_{ij} (g_s')_{pq} dV.
\end{align}
Taking another derivative,
\begin{align}
\begin{split}
\frac{d^2}{ds^2} (\mf_{\tau}[g_s]) &=
-2 \int (g_s')^{ip}  g^{jq} ( \nabla \mf_{\tau} (g_s)_{ij} (g_s')_{pq} dV
+ \int \langle (\nabla \mf_{\tau} )' (g_s) g_s', g_s' \rangle dV\\
&+ \int \langle  \nabla \mf_{\tau} (g_s), g_s'' \rangle dV
+ \int  \langle \nabla \mf_{\tau} (g_s), g_s' \rangle (dV)'.
\end{split}
\end{align}
If $\nabla \mf_{\tau} (g_s) = c g$, then this simplifies to \eqref{fdd}.
\end{proof}

The next result gives the relation between the second variation of the un-normalized and normalized functionals:
\begin{proposition}
\label{2vfpr}
If $g$ is Einstein with $\nabla \mf_{\tau} = c \cdot g$, then
\begin{align}
\label{2vf}
\begin{split}
(\tmf_{\tau}[g_s])''
&= -p(p+1) (Vol)^{p-2} \big( (Vol)' \big)^2 \cdot \mf_{\tau}\\
&+ (Vol)^p \Big\{ - c \int_M |g_s'|^2 dV
+  \int \langle (\nabla \mf_{\tau})' (g_s) g_s', g_s' \rangle dV \Big\}.
\end{split}
\end{align}

In particular, if $(\nabla \mathcal{F}_{\tau})'$ and $(\nabla \tilde{\mathcal{F}}_{\tau})'$ denote the linearized Euler-Lagrange equations of $\mathcal{F}_{\tau}$
and $\tilde{\mathcal{F}}_{\tau}$ respectively, then
\begin{align} \label{Ftp2}
(\nabla \tilde{\mathcal{F}}_{\tau})'h = (\nabla \mathcal{F}_{\tau})'h - c h.
\end{align}
\end{proposition}

\begin{proof}
Substituting \eqref{fdd} into \eqref{tdd}, and
using \eqref{vdd},  we have
\begin{align}
\begin{split}
(\tmf_{\tau} [g_s])''
&= -p(p+1) (Vol)^{p-2} \big( (Vol)' \big)^2 \cdot \mf_{\tau}
+ p (Vol)^{p-1} (Vol)'' \cdot \mf_{\tau} + (Vol)^p \cdot (\mf_{\tau})''\\
& =  -p(p+1) (Vol)^{p-2} \big( (Vol)' \big)^2 \cdot \mf_{\tau}
+ (Vol)^p \Big(  \int \langle (\nabla \mf_{\tau})' (g_s) g_s', g_s' \rangle dV \Big)
\\
& +  p (Vol)^{p-1} \cdot \mf_{\tau} \cdot
\Big( - \frac{1}{2} \int_M  | g_s'|^2 dV
+ \frac{1}{2} \int_M ( tr_{g_s} g_s'') dV + \frac{1}{4}
\int ( tr_{g_s} g_s')^2 dV \Big)\\
&+ c (Vol)^p \Big(  - 2 \int |g_s'|^2 dV + \int (tr_{g_s} g_s'') dV
+ \frac{1}{2} \int  (tr_{g_s} g_s')^2 dV \Big).
\end{split}
\end{align}
From \eqref{foc2}, we have that
\begin{align}
c = - \frac{p}{2} (Vol)^{-1} \cdot \mf_{\tau},
\end{align}
and substituting this into the above yields
\begin{align}
\begin{split} \label{Fpp22}
(\tmf_{\tau}[g_s])''
& =  -p(p+1) (Vol)^{p-2} \big( (Vol)' \big)^2 \cdot \mf_{\tau}
+  (Vol)^p \Big(  \int \langle (\nabla \mf_{\tau})' (g_s) g_s', g_s' \rangle dV \Big)\\
& +  p (Vol)^{p-1} \cdot \mf_{\tau} \cdot
\Big( - \frac{1}{2} \int_M  | g_s'|^2 dV
+ \frac{1}{2} \int_M ( tr_{g_s} g_s'') dV + \frac{1}{4}
\int ( tr_{g_s} g_s')^2 dV \Big)\\
& - \frac{p}{2} (Vol)^{p-1}\cdot \mf_{\tau} \cdot \Big( - 2 \int |g_s'|^2 dV +  \int (tr_{g_s} g_s'') dV
+ \frac{1}{2} \int  (tr_{g_s} g_s')^2 dV \Big),
\end{split}
\end{align}
whichs yields the claimed formula.
\end{proof}

\begin{remark} \label{Jrem}
{\em  At this point it is important to make a remark about the
use of the terms second variation, linearized Euler-Lagrange equation,
and Jacobi operator, and the amount of regularity necessary for these
to be well-defined.

First, the $L^2$-gradient of $\tmf_{\tau}$ calculated in Section \ref{LocalSec} is
clearly fourth order in the metric, since it is obtained by integrating by parts.
Consequently, the linearized Euler-Lagrange equation
is a fourth-order operator on tensor fields. We also refer to this
operator as the Jacobi operator. The second variation can be
expressed as a bilinear form on symmetric tensor fields,
as in (\ref{Fpp22}) above.

Since $\tmf_{\tau}$ only involves two derivatives of the metric,
the second variation is actually second order in the infinitesimal variation.
Therefore, although we will write $\int \langle Jh, h \rangle\ dV$
to denote the second variation (where $J$ is the Jacobi operator),
this expression actually makes sense as long as $h  \in C^{2}(S^2(T^{*}M))$.}
\end{remark}

\subsection{Jacobi operator}
\label{FSV}

In this subsection, we will compute an expression for the
linearized Euler-Lagrange equation of $\mathcal{F}_{\tau}$ for an arbitrary variation $h$ of
an Einstein metric; i.e., the operator $(\nabla \mathcal{F}_{\tau})'h$.
Subsequently, in Sections \ref{TTsec} and \ref{confsec}
we will examine the action of this operator on
transverse-traceless tensors and pure trace tensors.

The first step is to linearize the hessian of the Ricci tensor at an
Einstein metric:
\begin{proposition}
\label{hrlin}
If $g$ is Einstein,  satisfying $Ric = \frac{R}{n} g$, then with $h = g'$,
then
\begin{align}
(\n_l \n_j R_{kp})' =  \n_l \n_j (R_{kp})' - \frac{R}{n} \cdot \n_l \n_j h_{kp}.
\end{align}
\end{proposition}
\begin{proof}
We write out
\begin{align}
\n_l ( \n_j R_{kp}) = \partial_l ( \n_j R_{kp})
- \n_q R_{kp} \Gamma^q_{lj} - \n_j R_{qp} \Gamma^q_{lk}
- \n_j R_{kq} \Gamma^q_{lp},
\end{align}
and
\begin{align}
\n_i R_{jk} = \partial_i R_{jk} - R_{pk} \Gamma^p_{ij} - R_{jp} \Gamma^p_{ik}.
\end{align}
Next we linearize, with $g' = h$. We use tensorality  to simply toss out terms
which have Christoffel symbols, and replace derivatives with covariant
derivatives to obtain
\begin{align}
\begin{split}
(\n_l \n_j R_{kp})' &= \n_l \n_j (R_{kp})' - R_{qp} \n_l ( \Gamma^q_{jk})'
- R_{kq} \n_l ( \Gamma^q_{jp})'\\
& -  \n_q R_{kp} ( \Gamma^q_{lj})' - \n_j R_{qp} (\Gamma^q_{lk})'
- \n_j R_{kq} ( \Gamma^q_{lp})'.
\end{split}
\end{align}
Since Ricci is parallel, the last three terms vanish,
and using the formula
\begin{align}
(\Gamma^k_{ij})' = \frac{1}{2} g^{kl} ( \n_i h_{jl} + \n_j h_{il} - \n_l h_{ij}),
\end{align}
we obtain
\begin{align}
\begin{split}
\label{e1}
&(\n_l \n_j R_{kp})' = \n_l \n_j (R_{kp})'- R_{qp} \n_l
\left( \frac{1}{2} g^{qm} ( \n_j h_{km} + \n_k h_{jm} - \n_m h_{jk}  ) \right)\\
&- R_{kq} \n_l
\left( \frac{1}{2} g^{qm} ( \n_j h_{pm} + \n_p h_{jm} - \n_m h_{jp}  ) \right)
=  \n_l \n_j (R_{kp})' - \frac{R}{n} \cdot \n_l \n_j h_{kp}.
\end{split}
\end{align}
\end{proof}

With these formulas we can linearize the Euler-Lagrange equations of $\mf_0$:

\begin{theorem}
\label{ric2p}
If $g$ is Einstein, then
\begin{align}
\begin{split}
\label{Ricvar}
(\nabla \mf_0)'_{pq}
&= - \Delta (Ric)_{pq}' + \frac{R}{n} \Delta h_{pq}
-2 \Big(  \frac{R}{n} (Ric_{pq})' -  \frac{R}{n} R_{pkql} h^{kl}
+ R_{pkql} g^{ki} g^{lj} (Ric')_{ij} \Big) \\
& + \nabla^2_{pq} ( tr_g Ric') - \frac{R}{n} \nabla^2_{pq} (tr_g h)
- \frac{1}{2} \Big( \Delta ( tr_g Ric') - \frac{R}{n} \Delta(tr_g h)
\Big) g_{pq} \\
& + \Big( - \frac{R^2}{n^2} tr_g h +  \frac{R}{n} tr_g (Ric') \Big)g_{pq}
+ \frac{R^2}{2n} h_{pq}.
\end{split}
\end{align}
\end{theorem}

\begin{proof}
This is a consequence of the following formulae,
which follow from Proposition \ref{hrlin}:
\begin{align}
( \Delta Ric)' &= \Delta (Ric)' - \frac{R}{n} \Delta h,\\
( \nabla^2 R)' &= \nabla^2 ( tr_g Ric') - \frac{R}{n} \nabla^2 (tr_g h),\\
\big( (\Delta R) g \big)' & = \Big( \Delta ( tr_g Ric') - \frac{R}{n} \Delta(tr_g h)
\Big) g,\\
(|Ric|^2 g) ' &= \Big( -2 \frac{R^2}{n^2} tr_g h + 2 \frac{R}{n} tr_g (Ric') \Big)g
+ \frac{R^2}{n} h, \\
(R_{pkql}R^{kl} )' &= \frac{R}{n} (Ric_{pq})'
- \frac{R}{n} R_{pkql} h^{kl} + R_{pkql} g^{ki} g^{lj} (Ric')_{ij}.
\end{align}
\end{proof}

Next, we linearize the Euler-Lagrange equations of $\mathcal{S} = \int R^2$:
\begin{theorem}
\label{r2p}
If $g$ is Einstein, then
\begin{align}
\begin{split} \label{Rvar}
(\nabla \mathcal{S})'
&= 2 \Big( \nabla^2 (tr Ric') - \frac{R}{n} \nabla^2 (tr h) \Big)
- 2 \Big( \Delta (tr Ric')  - \frac{R}{n} \Delta (tr h) \Big) g \\
&- 2 \Big( \frac{R}{n} R' g + R Ric' \Big) + \frac{1}{2} ( 2 R R' g + R^2 h ).
\end{split}
\end{align}
\end{theorem}

\begin{proof}
This is a consequence of the formulae in the proof of Theorem \ref{ric2p},
together with the following:
\begin{align}
(R \cdot Ric)' &= \frac{R}{n} (R') g + R (Ric)'\\
(R^2 g)' &= 2 R R' g + R^2 h.
\end{align}
\end{proof}

To understand the variational properties of $\tilde{\mathcal{F}}_{\tau}$
we will rely on a splitting
formula which reduces the problem to checking
two types of variations:  those which are traceless and
divergence-free (transverse-traceless), and those which are
pure trace (conformal).  Therefore, we need to compute the
second variation (i.e., the Jacobi operator) on each type.

\subsection{Transverse-traceless variations}
\label{TTsec}

In this section, we assume that the variation $h$ satisfies
$\delta h = 0$ and $tr_g h = 0$. Define
\begin{align}
(\Delta_L h)_{ij} = \Delta h_{ij} + 2 R_{ipjq} h^{pq} - \frac{2}{n} R h_{ij},
\end{align}
and note that
\begin{align}
Ric'h = -\frac{1}{2} \Delta_L h.
\end{align}

\begin{theorem} \label{F0caseTT} In dimension $n$,
if $g$ is Einstein and $h$ is transverse-traceless, then
the linearized Euler-Lagrange equation of $ \tmf_0$ is
\begin{align}
\label{norf0}
(\nabla \tmf_0)'h =
\frac{1}{2} \Delta_L^2 h + \frac{3}{n} R \Delta_L h + \frac{4}{n^2} R^2 h.
\end{align}
\end{theorem}

\begin{proof}
If $h$ is transverse-traceless, then \eqref{Ricvar} simplifies to
\begin{align}
\begin{split}
(\nabla \mf_0)'_{pq}
&= - \Delta (Ric)_{pq}' + \frac{R}{n} \Delta h_{pq}\\
&-2 \Big(  \frac{R}{n} (Ric_{pq})' - \frac{R}{n} R_{pkql} h^{kl}
+ R_{pkql} g^{ki} g^{lj} (Ric')_{ij} \Big) + \frac{R^2}{2n} h_{pq}.
\end{split}
\end{align}
We collect terms as follows
\begin{align}
\Big(- \Delta (Ric)_{pq}' - 2  R_{pkql} g^{ki} g^{lj} (Ric')_{ij}\Big)
+ \frac{R}{n} \Big( \Delta h_{pq} - 2 (Ric_{pq})'
 + 2 R_{pkql} h^{kl} \Big)
+  \frac{R^2}{2n} h_{pq}.
\end{align}
The first term is
\begin{align}
- \Delta (Ric)_{pq}' - 2  R_{pkql} g^{ki} g^{lj} (Ric')_{ij}
= - \Delta_L (Ric)_{pq}' - \frac{2}{n} R (Ric')_{pq}.
\end{align}
Using this, the entire expression is rewritten as
\begin{align}
 - \Delta_L (Ric)_{pq}' + \frac{R}{n} \Big( \Delta h_{pq} - 4 (Ric_{pq})'
 + 2 R_{pkql} h^{kl} \Big)
+  \frac{R^2}{2n} h_{pq}.
\end{align}
In the second term we have the expression
\begin{align}
\begin{split}
 \Delta h_{pq} - 4 (Ric_{pq})'
 + 2 R_{pkql} h^{kl} &=  \Delta h_{pq} + 2 R_{pkql} h^{kl} - 4 Ric_{pq}'\\
& = \Delta_L h_{pq} + \frac{2}{n} R h_{pq} + 2 \Delta_L h_{pq}
= 3 \Delta_L h_{pq} + \frac{2}{n} R h_{pq}.
\end{split}
\end{align}
Converting $Ric'$ into $\Delta_L$, we obtain
\begin{align}
\label{unnf0}
\frac{1}{2} \Delta_L^2 h_{pq} + \frac{3}{n} R \Delta_L h_{pq}
+ \frac{n+4}{2n^2} R^2 h_{pq}.
\end{align}
Finally, for the normalized functional, we must add to
\eqref{unnf0} the correction term from Proposition \ref{2vfpr}, $- c h$,
with the constant $c$ determined in Proposition \ref{Ricc}, which
yields \eqref{norf0}.
\end{proof}
For $\mathcal{S}$, there is a degeneracy, which is reflected in
the following theorem in which the linearized Euler-Lagrange
equations restricted to transverse-traceless are seen to be second order
(instead of the expected fourth order).
\begin{theorem} \label{ScaseTT} In dimension $n$,
if $g$ is Einstein and $h$ is transverse-traceless, then
the linearized Euler-Lagrange equation of $\tilde{\mathcal{S}}$ is
\begin{align}
\label{nors0}
(\nabla \tilde{\mathcal{S}})'h = R \Delta_L h + \frac{2}{n} R^2 h.
\end{align}
\end{theorem}

\begin{proof}
The only terms which will contribute anything from Theorem \ref{r2p} are
\begin{align}
(-2  R Ric)' = R \Delta_L h,
\end{align}
and the term
\begin{align}
\frac{1}{2} (R^2 g)' = \frac{1}{2} R^2 h,
\end{align}
so that we obtain
\begin{align}
\label{unns0}
(\nabla {\mathcal{S}})'h = R \Delta_L h + \frac{1}{2} R^2 h.
\end{align}
For the normalized functional, we must add to
\eqref{unns0} the correction term from Proposition \ref{2vfpr}, $- c h$,
with the constant $c$ determined in Proposition \ref{Rc}, which
yields \eqref{nors0}.
\end{proof}

 Combining Theorems \ref{F0caseTT} and \ref{ScaseTT}, we obtain
a useful factorization of the linearized Euler-Lagrange equations.
\begin{theorem} \label{JTT} In dimension $n$,
if $g$ is Einstein, and $h$ is transverse-traceless, then
the linearized Euler-Lagrange equation of $\tilde{\mathcal{F}}_{\tau}$ is
\begin{align}
\label{ttnvar}
(\nabla \tilde{\mathcal{F}}_{\tau})'h &=
\frac{1}{2} \Big( \Delta_L + \frac{2}{n} R \Big)
\Big( \Delta_L + \Big( \frac{4}{n} + 2\tau \Big) R \Big) h.
\end{align}
\end{theorem}

\begin{remark}{\em  Notice that for $\tau = -(1/n)$, the operator is a square,
\begin{align}
 \frac{1}{2} \Big( \Delta_L + \frac{2}{n} R \Big)^2 h.
\end{align}
This is not a surprise, since the functional in this
case is $\int |E|^2$,  and the kernel consists exactly of
the infinitesimal Einstein deformations, i.e., transverse-traceless
tensors satisfying $E' h = 0$.}
\end{remark}

\subsection{Conformal Variations}
\label{confsec}

Next, we consider conformal variations of the functionals $\tmf_0$ and
$\tilde{\mathcal{S}}$
at an Einstein metric.
\begin{theorem}
\label{conff0}
If $h = f g$, and $g$ is Einstein, then the trace of the
Jacobi operator of $\tmf_0$ is given by
\begin{align} \label{Jfc}
tr(Jf) = tr[J (fg)] = \frac{n(n-1)}{2} \Delta^2 f - \frac{ n^2 - 10 n + 8 }{2n} R \Delta f
- \frac{n-4}{n} R^2 f.
\end{align}
\end{theorem}

\begin{remark} {\em
As we will see below, we only need a formula for the trace of $J$, since we will
be viewing $J$ as a bilinear form restricted to variations of the form $h = fg$, and
\begin{align*}
\langle J(fg), fg \rangle_{L^2} = \langle tr(Jf), f \rangle_{L^2}.
\end{align*} }
\end{remark}

\begin{proof}
If $h = f g$, then we have
\begin{align}
\label{cnf1}
Ric' &= - \frac{1}{2} ( \Delta f g + (n-2) \n_i \n_j f)\\
\label{cnf2}
tr(Ric') & = (1-n) \Delta f.
\end{align}
Using Theorem \ref{ric2p}, we compute
\begin{align}
\begin{split}
tr (\nabla \mf_0)'(fg)
& = \frac{n(n-1)}{2} \Delta^2 f - \frac{ n^2 - 10 n + 8 }{2n} R  \Delta f
- \frac{n-4}{2n} R^2 f.
\end{split}
\end{align}
For the normalized functional, from Proposition \ref{2vfpr}
and Proposition \ref{Ricc}, we must add the term
\begin{align}
- c f = -\frac{n-4}{2n} f,
\end{align}
and the result follows.
\end{proof}

The analogous result for $\tilde{\mathcal{S}}$ is the following.
\begin{theorem}
\label{confs0}
If $h = f g$, and $g$ is Einstein, then the trace of the Jacobi operator
of  $\tilde{\mathcal{S}}$ is given by
\begin{align}
tr(J f) = 2(n-1)^2 \Delta^2 f - (n-6)(n-1) R \Delta f
- (n-4) R^2 f.
\end{align}
\end{theorem}
\begin{proof}
In addition to \eqref{cnf1} and \eqref{cnf2}, we
note the formula
\begin{align}
R' = (1-n) \Delta f - R f.
\end{align}
Using Theorem \ref{r2p}, we compute
\begin{align}
\begin{split}
tr [ \mathcal{S}'(fg)]
& = 2(n-1)^2  \Delta^2 f - R(n-6)(n-1)  \Delta f + \frac{4-n}{2} R^2 f.
\end{split}
\end{align}
Restricting to $\mathcal{M}_1$, from Propositions \ref{2vfpr}
and \ref{Ricc}, we must add the
term
\begin{align}
- c tr\ h = - \frac{n-4}{2n}tr\ h =  - \frac{n-4}{2n}f,
\end{align}
and the result follows.
\end{proof}
Combining Theorems \ref{conff0} and \ref{confs0}, we obtain
a useful factorization of the Jacobi operator.
\begin{theorem} \label{FtJf}
If $h = f g$, and $g$ is Einstein, then the trace of the Jacobi operator of
$\tilde{\mathcal{F}}_{\tau}$ is given by
\begin{align}
\label{confjac}
tr(J f) &=
\frac{1}{2n}  \Big( (n-1) \Delta + R \Big)
\Big( n(n - 4\tau + 4n \tau) \Delta - 2(n-4)(1 + n\tau) R \Big)f.
\end{align}
\end{theorem}

\section{Local variational properties of Einstein metrics}

In this section we apply the formulas of the preceding section to determine the variational properties of $\tilde{\mathcal{F}}_{\tau}$ near Einstein metrics.  As in Section \ref{SV}, we consider transverse-traceless and conformal deformations separately.

\subsection{Transverse-traceless variations}
We begin with the positive scalar curvature case.
\begin{theorem} \label{TTPosCor} Let $(M,g)$ be an $n$-dimensional Einstein manifold with $R > 0$.

\noindent $(i)$  If $\tau > -1/n$ and
\begin{align} \label{lt1}
\mbox{spec}_{TT}(-\Delta_L) \cap \Big[ \frac{2}{n}R ,
\Big(\frac{4}{n} + 2\tau \Big) R \Big] = \varnothing,
\end{align}
then the Jacobi operator of $\tilde{\mathcal{F}}_{\tau}$ restricted to transverse-traceless symmetric tensors is positive-definite:
\begin{align}
\label{posdefj}
\int \langle Jh,h \rangle\ dV = \int \langle (\nabla \tilde{\mathcal{F}}_{\tau})' h, h \rangle\ dV \geq \epsilon_0 \int |h|^2\ dV,
\end{align}
for some $\epsilon_0 > 0$ and all transverse-traceless $h$.   \vspace{1mm}

\noindent $(ii)$  If $\tau < -1/n$ and
\begin{align} \label{lt2}
\mbox{spec}_{TT}(-\Delta_L) \cap  \Big[ \Big( \frac{4}{n} + 2\tau \Big)R ,
\frac{2}{n}R \Big] = \varnothing,
\end{align}
then the same result holds.
\end{theorem}

\begin{proof}
Let $h$ be an eigensection of $(-\Delta_L)$ with eigenvalue $\mu_L$.
Since $(M,g)$ is Einstein, by \cite{Lichnerowicz} $\Delta_L$ maps the space of
transverse-traceless ($TT$) tensors to itself; hence we may assume $h$ is $TT$.

From \eqref{ttnvar}, we must consider the polynomial
\begin{align}
\frac{1}{2} \Big( - \mu_L + \frac{2}{n} R \Big)
\Big( - \mu_L + \Big( \frac{4}{n} + 2\tau \Big) R \Big) h.
\end{align}
It is then obvious if $\mu_L$ is outside of the stated range,
then this is strictly positive.
\end{proof}

A similar argument (with the signs simply reversed) gives

\begin{theorem} \label{TTNegCor} Let $(M,g)$ be an $n$-dimensional Einstein manifold with $R < 0$.

\noindent $(i)$  If $\tau > -1/n$ and
\begin{align} \label{lt3}
\mbox{spec}_{TT}(-\Delta_L) \cap \Big[ \Big(\frac{4}{n} + 2\tau\Big)R , \frac{2}{n}R \Big]
= \varnothing,
\end{align}
then \eqref{posdefj} holds for some $\epsilon_0 > 0$
and all transverse-traceless $h$.

\noindent $(ii)$  If $\tau < -1/n$ and
\begin{align} \label{lt4}
 \mbox{spec}_{TT}(-\Delta_L) \cap \Big[ \frac{2}{n}R, \Big(\frac{4}{n} + 2\tau \Big) R \Big]
= \varnothing,
\end{align}
then the same result holds.
\end{theorem}

\subsection{Conformal variations} \label{ConfSS}
We will let $\mathcal{M}_1([g])$ denote the space of
unit volume metrics conformal to $g$. The tangent space of $\mathcal{M}_1([g])$
consists of functions with mean value zero.

For the general quadratic functional $\tilde{\mathcal{F}}_{\tau}$ we have

\begin{theorem} \label{conftPos} Let $(M,g)$ be an Einstein manifold of dimension $n > 2$ with $R > 0$.

\noindent $(i)$  Let $n=3$.  If $\tau > -3/8$,
then $g$ is a local {\em minimizer} of $\tilde{\mathcal{F}}_{\tau}$ restricted to $\mathcal{M}_1([g])$.  If $\tau < -5/12$, then $g$ is a
local {\em maximizer}.

\noindent $(ii)$  Let $n=4$.  If $\tau > -1/3$, then $g$ is a local {\em minimizer} $\tilde{\mathcal{F}}_{\tau}$ restricted to $\mathcal{M}_1([g])$.  If $\tau < -1/3$, then $g$ is a local {\em maximizer}. If $\tau = -1/3$, then $\tilde{\mathcal{F}}_{\tau}$ is conformally invariant.

\noindent $(iii)$  Let $n \geq 5$.  If $\tau >   \frac{4 - 3n}{2 n (n-1)}$, then $g$ is a local {\em minimizer} of $\tilde{\mathcal{F}}_{\tau}$ restricted
to  $\mathcal{M}_1([g])$.   If $\tau <  - \frac{n}{4 (n-1)}$, then $g$ is a local {\em maximizer}.

Furthermore, all of the above extrema are strict if $g$ is not isometric to the round sphere.

\end{theorem}

\begin{proof}Since we are restricting to $\mathcal{M}_1([g])$, it
suffices to prove that the second variation of the functional on
functions with mean value zero is strictly positive (or strictly
negative in the maximizing case).

Consider the factorization of the Jacobi operator given in \ref{confjac}.
Let $\lambda$ denote a non-zero eigenvalue of $(-\Delta)$,
and consider the polynomial
\begin{align} \label{ptdef}
p_{\tau}(\lambda) = \frac{1}{2n}  \Big( (n-1) \lambda - R \Big)
\Big( n(n - 4\tau + 4n \tau) \lambda + 2(n-4)(1 + n\tau) R \Big).
\end{align}
The eigenvalue estimate of Lichnerowicz says that
\begin{align*}
\lambda_1 \geq R/(n-1),
\end{align*}
with strict inequality if $g$ is not isometric to $S^n$ \cite[Theorem 4.19]{Aubin}.
Consequently, the sign of $p_{\tau}$ is determined by the second factor:
the functional will be minimizing, for example, if
\begin{align} \label{2term}
 n(n - 4\tau + 4n \tau) \lambda + 2(n-4)(1 + n\tau) R > 0.
\end{align}
Note that if
\begin{align}
\label{treq}
\tau > - \frac{n}{4(n-1)} \equiv \tau_2,
\end{align}
then the coefficient of $\lambda$ in (\ref{2term}) is positive, so we
can use the Lichnerowicz estimate again to find
\begin{align} \label{lpoly}
 n(n - 4\tau + 4n \tau) \lambda + 2(n-4)(1 + n\tau) R
\geq \Big(  \frac{  n(n - 4\tau + 4n \tau)}{n-1} +  2(n-4)(1 + n\tau)
\Big) R.
\end{align}
Some algebra shows that this is
positive for
\begin{align} \label{tau1def}
\tau > \frac{4 - 3n}{2 n (n-1)} \equiv \tau_1.
\end{align}

When $n = 3$, then $\tau_2 = -3/8 > \tau_1 = -5/12$, so the second factor in (\ref{2term})
is positive provided $\tau > -3/8.$
If $n=4$, then $\tau_2 = \tau_1 = -1/3$, so the second factor is positive for
$\tau > -1/3$.  When $n \geq 5$, then $\tau_1 > \tau_2$ and
the second factor is positive for $\tau > \tau_1$.

If \begin{align}
\label{treq2}
\tau < - \frac{n}{4(n-1)} = \tau_2,
\end{align}
then the sign of the coefficient in of $\lambda$ in (\ref{2term}) is reversed,
and we find that the second factor is negative provided
\begin{align*}
\tau < \frac{4 - 3n}{2 n (n-1)} = \tau_1.
\end{align*}
If $n=3$, it follows that the second factor is negative for $\tau < \tau_1 = -5/12 < \tau_2$; when $n=4$, the second factor is negative for $\tau < -1/3$.
When $n \geq 5$, then as $\tau_2 < \tau_1$, we get negativity for $\tau < \tau_2$.
\end{proof}

In the negative case, because we lack an estimate for the first eigenvalue in terms of
the scalar curvature, stability often depends on an estimate of $\lambda_1$:

\begin{theorem} \label{conftNeg} Let $(M,g)$ be an Einstein manifold of dimension $n > 2$ with $R < 0$.

\noindent $(i)$  Let $n=3$.  If $\tau > -1/3$, then $g$ is a (strict) local {\em minimizer} of $\tilde{\mathcal{F}}_{\tau}$ restricted to  $\mathcal{M}_1([g])$.  If $\tau < -3/8$, then $g$ is a (strict) local {\em maximizer}.

\noindent $(ii)$  Let $n=4$.  If $\tau > -1/3$, then $g$ is a strict local {\em minimizer} $\tilde{\mathcal{F}}_{\tau}$ restricted to  $\mathcal{M}_1([g])$, and is a strict local {\em maximizer} if $\tau < -1/3$.

\noindent $(iii)$  Let $n \geq 5$.

$\bullet$  If $\tau > -1/n$, then $g$ is a strict local {\em minimizer} of $\tilde{\mathcal{F}}_{\tau}$ restricted
to  $\mathcal{M}_1([g])$ provided
\begin{align} \label{ell1}
\lambda_1 > \frac{(n-4)}{2(n-1)}(-R),
\end{align}
where $\lambda_1 = \lambda_1(-\Delta)$.

$\bullet$  If $- \frac{n}{4 (n-1)} < \tau < -1/n$, then $g$ is a strict local {\em minimizer} of $\tilde{\mathcal{F}}_{\tau}$ restricted to  $\mathcal{M}_1([g])$.

$\bullet$ If $\tau <  - \frac{n}{4 (n-1)}$, then $g$ is a strict local {\em maximizer} of $\tilde{\mathcal{F}}_{\tau}$ restricted to  $\mathcal{M}_1([g])$ provided \eqref{ell1} holds.
\end{theorem}

\begin{proof} For the same reasons as above, we restrict to the space of
functions with mean value zero.  For $R < 0$ the first term in the factorization of $p_{\tau}$ in (\ref{ptdef}) is always positive, so once again the
extremizing properties of the functional are determined by the sign of the second term
\begin{align} \label{qtdef}
q(\lambda) =  n(n - 4\tau + 4n \tau) \lambda + 2(n-4)(1 + n\tau) R.
\end{align}

First, consider the case $n \geq 5$.  If $\tau > -1/n > - \frac{n}{4 (n-1)}$, then the coefficient of $\lambda$ in (\ref{qtdef}) is positive, and $q(\lambda) > 0$ provided
\begin{align} \label{Lbig}
\lambda > \frac{2(n-4)(1+n\tau)}{n(n-4\tau+ 4n\tau)}(-R).
\end{align}
Note the expression on the right-hand side is increasing as a function of $\tau$, and tends (as $\tau \to \infty$) to $\frac{(n-4)}{2(n-1)}(-R)$.  Therefore, if (\ref{ell1}) holds, then $q > 0$ and $g$ is a
(strict) local minimizer.  If $-\frac{n}{4(n-1)} < \tau < -1/n$, then the coefficient of $\lambda$ in (\ref{qtdef}) is still positive, and (\ref{Lbig}) always holds (since the right-hand side is negative); i.e., $g$ is
a local minimizer.

If $\tau < - \frac{n}{4(n-1)}$, then the coefficient of $\lambda$ in (\ref{qtdef}) is negative, and $q(\lambda) < 0$ provided (\ref{Lbig}) holds.  Hence, in this case $g$ is a (strict) local maximizer.  For the cases $n = 3$ and $n=4$, we argue as we did in the proof of Theorem \ref{conftPos}.
\end{proof}


 The analogous results for $\tilde{\mathcal{S}}$ hold by formally setting
$\tau = \infty$ in the above arguments.

\begin{proposition} If the dimension $n \geq 3$, then an Einstein metric with $R > 0$ is a local
minimizer of $\tilde{\mathcal{S}}$ restricted to
$\mathcal{M}_1([g])$, and is a strict local minimizer provided
that $g$ is not isometric to $S^n$.

If the dimension $n = 3$ or $4$, then an Einstein metric with $R < 0$ is a (strict)
local minimizer of $\tilde{\mathcal{S}}$ restricted to
$\mathcal{M}_1([g])$.

If $n \geq 5$ then an Einstein metric with $R < 0$ is a (strict)
local minimizer of $\tilde{\mathcal{S}}$
restricted to $\mathcal{M}_1([g])$ provided
\begin{align} \label{LLower2}
\lambda_1 > \frac{(n-4)}{2(n-1)}(-R).
\end{align}
\end{proposition}

\begin{proof} The details are similar to above, the proof is omitted.
\end{proof}

\section{Rigidity for Einstein metrics} \label{RigidSec}

In this section we study the rigidity of Einstein metrics as critical points of $\tmf_{\tau}$.  Recall by Lemma \ref{H1Lemma},
\begin{align*}
H^1_{\tau} = \big\{ h \in \overline{S}_{0}^2(T^{*}M)\ \big|\ (\nabla \tmf_{\tau})_g'h = 0,\ \beta_g h = 0 \big\}.
\end{align*}
As a preliminary result we give conditions which imply that elements of $H^1_{\tau}$ are transverse-traceless:

\begin{lemma} \label{H1TT}
Let $(M,g)$ be an $n$-dimensional Einstein manifold which is different from the round sphere.  Assume $\tau \neq -\frac{n}{4(n-1)}$.

If
\begin{align} \label{H1gap}
\frac{2(n-4)(1 + n\tau) }{(-n^2 + 4n\tau - 4n^2 \tau) }R \notin \mbox{spec}(-\Delta),
\end{align}
then every $h \in H^1_{\tau}$ is transverse-traceless.
\end{lemma}

\begin{remark} {\em
For the round sphere, if $\psi_1$ is a first-order spherical harmonic, then $h = \psi_1 g$ is in $H^1_{\tau}$.
However, these deformations are
not essential, as they arise from paths of conformal diffeomorphisms.
These can be dealt with by a modified ``slicing'' procedure,
see Lemma \ref{sliceSn} below.}
\end{remark}

\begin{proof}
Let $z$ denote the trace-free part of $h$ and write
\begin{align*}
h = z + f g,
\end{align*}
where $f = \frac{1}{n}(tr\ h).$  Note that $\beta_g h = 0$ implies that $z$ is transverse-traceless.

We first claim that
\begin{align} \label{Jclaim}
tr(Jh) = 0 \Rightarrow tr (Jf) = 0,
\end{align}
where $J = (\nabla \tilde{\mathcal{F}}_{\tau})'$ denotes the Jacobi operator. To see this, let $\psi \in C^{\infty}(M)$.  Since $Jh = 0$,
\begin{align*}
0 &= \big\langle tr(Jh), \psi  \big\rangle_{L^2}
= \big\langle Jh, \psi \cdot g \big\rangle_{L^2}
= \big\langle Jz + J(fg), \psi g \big\rangle_{L^2} \\
&= \big\langle Jz, \psi g \big\rangle_{L^2}
+ \big\langle J(fg), \psi g \big\rangle_{L^2} = \big\langle tr(Jz), \psi \big\rangle_{L^2} + \big\langle tr(Jf), \psi \big\rangle_{L^2}.
\end{align*}
One can check (see Theorem \ref{JTT}) that $tr(Jz) = 0$, since $z$ is trace-free.  Therefore,
\begin{align*}
0 = \big\langle tr(Jf), \psi \big\rangle_{L^2}
\end{align*}
for each $\psi \in C^{\infty}(M)$, which proves the claim.  Consequently, if $Jh = 0$ then by (\ref{confjac}) we have
\begin{align*}
0 = tr(J f) &=
\frac{1}{2n}  \Big( (n-1) \Delta + R \Big)
\Big( n(n - 4\tau + 4n \tau) \Delta - 2(n-4)(1 + n\tau) R \Big)f.
\end{align*}
Since $(M,g)$ is not the round sphere, the first operator in the above factorization has trivial kernel.  If $\tau \neq -\frac{n}{4(n-1)}$
and (\ref{H1gap}) holds,
then the second operator also has trivial kernel, and we conclude $f = 0$.
\end{proof}

\begin{proposition}  \label{H1Prop1}  With the same assumptions as Lemma \ref{H1TT}, if we assume in addition that
\begin{align} \label{DLgapH}
\Big\{ \frac{2}{n} R,  \Big( \frac{4}{n} + 2 \tau\Big)R  \Big\}
\notin \mbox{spec}_{TT} ( - \Delta_L),
\end{align}
then $H^1_{\tau} = 0$.
\end{proposition}

\begin{proof}
This follows directly from Lemma \ref{H1TT} and the
formula \eqref{ttnvar}.
\end{proof}

Therefore, to verify rigidity we need need to check two spectral conditions: one for the Laplacian on functions, the
other for the Lichnerowicz Laplacian on transverse-traceless tensors.  The following is a summary of when the condition
on $\mbox{spec}(-\Delta)$ holds:

\begin{lemma} \label{specgapLemma} Let $(M,g)$ be an Einstein manifold.  Assume one of
the following holds:

\noindent $(i)$ $R > 0$, $g$ is not round, and either

$\bullet$ $n =3$ and $ \tau\notin ( -5/12, - 3/8)$;

$\bullet$ $n = 4$ and $\tau \neq -1/3$; or

$\bullet$ $n \geq 5$ and
\begin{align}
\tau \notin \Big( - \frac{n}{4(n-1)} , \frac{4 - 3n}{2n(n-1)} \Big).
\end{align}

\noindent $(ii)$  $g$ is Ricci-flat, and $\tau \neq - \frac{n}{4(n-1)}$.

\noindent $(iii)$  $R < 0$, and either

$\bullet$ $n =3$ and  $ \tau \notin ( -3/8, - 1/3)$;

$\bullet$ $n = 4$ and $ \tau \neq -1/3$; or

$\bullet$ $n \geq 5$ and
\begin{align*}
\tau \in \Big(  - \frac{n}{4(n-1)} , -\frac{1}{n} \Big).
\end{align*}
Then (\ref{H1gap}) holds.
\end{lemma}

\begin{proof}  This follows immediately from the arguments in Section \ref{ConfSS}, where we considered the
conformal Jacobi operator.
\end{proof}

Summarizing the preceding results,
\begin{theorem} \label{h1thm}
Let $(M,g)$ be an $n$-dimensional Einstein manifold with
\begin{align}
\label{condh1}
\Big\{ \frac{2}{n} R,  \Big( \frac{4}{n} + 2 \tau \Big) R  \Big\}
\notin \mbox{spec}_{TT} ( - \Delta_L).
\end{align}
If any of the cases appearing in Lemma \ref{specgapLemma} holds, then $H^1_{\tau} = 0$.
If $n=4$ and $\tau = -1/3$, then under the assumption \eqref{condh1},
$H^1_{-1/3} = \{ 0 \}$.
\end{theorem}
\begin{proof}
This follows from the above, with the added remark that
in the case of $n=4$ and $\tau = -1/3$ (the functional $\tilde{\mathcal{W}}$),
the space $H^1_{-1/3}$ by definition contains only TT elements, so
only the condition on $\mbox{spec}_{TT}$ is needed.
\end{proof}

\section{Stability for critical metrics}  \label{StableSec}

In this section we consider stability properties of critical metrics,
and give the proof of Theorems \ref{StableMain}, \ref{Bachthm}, and \ref{RevBishop}.
First, we recall the definition of stability:

\begin{definition}
{\em
A metric $g$ is called {\em{strictly stable}} if
there is an $\epsilon_0 > 0$ such that the Jacobi operator $J$ associated to $\tilde{\mathcal{F}}_{\tau}$ satisfies
\begin{align} \label{1pos}
\int \langle Jh, h \rangle\ dV \geq \epsilon_0 \int |h|^2\ dV
\end{align}
for $h \in \mathcal{V} \equiv
\{ h \in C^{2,\alpha}(S^2(T^{*}M))\ \vert\ \beta_g h = 0 \}$.

If all kernel elements are integrable and \eqref{1pos} holds for some $\epsilon_0 > 0$ and all $h \in \mathcal{V}$ orthogonal to the kernel, then $g$ is called {\em{stable}}.}
\end{definition}
Note that in the following proposition, we do not assume $g$ is Einstein.
\begin{proposition}  \label{locmin}
Let $(M,g)$ be critical for $\tmf_{\tau}$. If $g$ is stable, then $g$ is a local
minimizer for $\tmf_{\tau}$. If $g$ is strictly stable, then $g$ is a
strict local minimizer for $\tmf_{\tau}$.
\end{proposition}

\begin{proof}
The proof begins with a modified ``slicing'' lemma.  Assume $\dim (H^1_{\tau}) = d > 0$,
and choose an $L^2$-orthonormal basis $\{ u_1, \dots, u_d \}$ for $H^1_{\tau}$.
Under the assumption that all kernel elements are integrable, there is a
$d$-dimensional family of critical metrics $g(Q)$ for
$Q \in \mathcal{B} = B(0, \epsilon) \subset \RR^d$,
depending smoothly on $Q$,
such that $\frac{\partial}{\partial Q_i} g(Q) |_{Q = 0} = u_i$
for $i = 1 \dots d$, and $g(0) = g$. The notation $|Q|$ will denote
the Euclidean norm of $Q$.

\begin{lemma}  \label{sliceSn}  For each metric $g_1 = g + \theta_1$ in a
sufficiently small $C^{\ell + 1,\alpha}$-neighborhood of $g$
($\ell \geq 1$), there exists a $C^{\ell + 2,\alpha}$-diffeomorphism
$\phi : M \rightarrow M$, a point $Q \in \mathbb{R}^d$,
and a constant $c \in \mathbb{R}$ such that
\begin{align} \label{pullscale10}
e^c \phi^{*}g_1 = g(Q) + \tilde{\theta},
\end{align}
with
\begin{align} \label{kkerSn}
\beta_{g} \tilde{\theta} = 0,
\end{align}
and
\begin{align} \label{trconSn} \begin{split}
\int tr_{g} \tilde{\theta}\ dV &= 0,  \\
\int \langle u_{\nu} , \tilde{\theta} \rangle \  dV &= 0, \ \ 1 \leq \nu \leq d.
\end{split}
\end{align}
Moreover, we have the estimates
\begin{align} \label{Qsmall}
| Q | \leq C \| \theta_1 \|_{C^{\ell +1,\alpha}},
\end{align}
\begin{align} \label{tilsmall2}
\| \tilde{\theta} \|_{C^{\ell+1,\alpha}} \leq C \| \theta_1 \|_{C^{\ell+1,\alpha}},
\end{align}
where $C = C(g)$.
\end{lemma}

\begin{proof}
Let $\{ \omega_1, \dots , \omega_{\kappa} \}$ denote a basis of the space of conformal Killing forms with respect to $g$.
Define the map
\begin{align}
\mathcal{H} : C^{\ell + 2,\alpha}(TM) \times \mathbb{R} \times \mathbb{R}^{d} \times \mathbb{R}^{\kappa} \times C^{\ell + 1,\alpha}(S^2(T^{*}M)) \rightarrow C^{\ell,\alpha}(T^{*}M) \times \mathbb{R} \times \mathbb{R}^{d}
\end{align}
 by
\begin{align} \label{Hdef} \begin{split}
\mathcal{H}(X,r,&Q,v,\theta) = \mathcal{H}_{\theta}(X,r,Q,v) \\
& \hskip-.25in = \Big( \beta_{g} \Upsilon(X,r,Q) + \sum_i v_i \omega_i, \int tr_{g} \Upsilon(X,r,Q)\ dV,
\int \Big \langle u_{\nu}, \Upsilon(X,r,Q) \Big\rangle dV \Big),
\end{split}
\end{align}
where
\begin{align} \label{Udef}
\Upsilon(X,r,Q) = e^r \phi^{*}_{X,1}(g + \theta) - g(Q),
\end{align}
and the final entry in (\ref{Hdef}) is understood to be an $d$-tuple.  If we endow the Euclidean factors in the
domain of $\mathcal{H}$ with the Euclidean norm, it is clear that $\mathcal{H}$ is a continuously differentiable mapping.

To compute the linearization
at $(X,r,Q,v,\theta) = (0,0,0,0,0)$, first note that
\begin{align*}
\Upsilon'(Y,s,P) &= \frac{d}{d\epsilon} \Upsilon(\epsilon Y, \epsilon s, \epsilon P) \Big|_{\epsilon = 0} \\
&= s g + \mathcal{L}_{g}[Y] - \sum_{m} p_m u_m.
\end{align*}
Therefore,
\begin{align} \label{LinHS} \begin{split}
\mathcal{H}_0'(Y,s,P,a) &= \Big( \Box_{\mathcal{K}} Y^{\flat}
- \beta_{g} ( \sum p_m u_m )  + \sum_i a_i \omega_i, n Vol(g) s, \\
& \hskip.5in  \int \big \langle u_{\nu}, \mathcal{L}_{g}Y \big\rangle\ dV
-  \int \Big\langle u_{\nu},  \Big( \sum_m  p_m u_m \Big) \Big \rangle dV \Big) \\
&= \Big( \Box_{\mathcal{K}} Y^{\flat} + \sum_i a_i \omega_i, n Vol(g) s
, \frac{2}{n} \int (tr\ u_{\nu})(\delta_{g} Y)\ dV - p_{\nu} \Big).
\end{split}
\end{align}
Here we have used the fact that
\begin{align*}
\beta_{g}( u_{\nu} ) = 0, \ \ \int (tr\ u_{\nu})\ dV = 0.
\end{align*}
It follows that the adjoint operator is given by
\begin{align} \label{HstarS}
(\mathcal{H}_0')^{*}(\eta,r,q_{\mu}) =
\big( \Box_{\mathcal{K}} \eta - \frac{2}{n} \sum_{m} q_m d(tr\ u_m),
n Vol(g) r, - q_{\mu}, \int \langle \eta, \omega_i \rangle\ dV \big).
\end{align}
Setting $(\mathcal{H}_0')^{*}(\eta,r,q_{\mu}) = 0$ we immediately find (by looking at the middle two arguments) that $r = 0$ and $q_{\mu} = 0, 1 \leq \mu \leq d$.
Therefore, the cokernel condition becomes
\begin{align*}
\Box_{\mathcal{K}} \eta &= 0, \\
 \int \langle \eta, \omega_i \rangle\ dV & =0, \ 1 \leq i \leq \kappa.
 \end{align*}
The first equation implies that $\eta$ is a conformal Killing field, while the second implies that $\eta$ is orthogonal (in $L^2$) to the space
of conformal Killing forms.  It follows that $\eta = 0$, so the map $\mathcal{H}_0^{\prime}$ is surjective.  Therefore, if $\theta_1$ is small we can find
$(X,c,Q,v)$ small with
\begin{align*}
\mathcal{H}_{\theta_1}(X,c,Q,v) = (0,0,0).
\end{align*}
In particular,
\begin{align} \label{bisz1}
\beta_g \Upsilon(X,c,Q)  + \sum_i v_i \omega_i = 0.
\end{align}
Fix $1 \leq \ell \leq \kappa$; if we pair both sides of (\ref{bisz}) with $\omega_{\ell}$ in $L^2$ we find $v_{\ell} = 0$, hence
\begin{align} \label{bisz}
\beta_g \Upsilon(X,c,Q) = 0.
\end{align}
Consequently, if we define $\tilde{\theta}$ by
\begin{align} \label{ttdef}
\tilde{\theta} =  \Upsilon(X,c,Q) = e^c \phi^{*}(g + \theta_1) - g(Q),
\end{align}
where $\phi = \phi_{X,1}$, then $\tilde{\theta}$ clearly satisfies the
conclusions of the Lemma.

\end{proof}


Before we continue the proof of Proposition \ref{locmin}, we remark that
there is an equivalent way of formulating the stability condition (\ref{1pos}).
In place of  (\ref{1pos}),
we could require that for $h \in \mathcal{V}$.
\begin{align} \label{1pos2}
\int \langle Jh, h \rangle\ dV \geq \epsilon_1 \Vert h \Vert_{H^2}^2,
\end{align}
where
\begin{align}
 \Vert h \Vert_{H^2}^2 = \int |\nabla^2 h|^2 dV + \int |\nabla h|^2 dV
+ \int |h|^2 dV. \end{align}
We refer to (\ref{1pos2}) as {\em $H^2$-stability}.
Clearly, $H^2$-stability implies stability.
However, the converse is also
true: Since $J$ is an elliptic operator on $\mathcal{V}$,
we can use the spectral decomposition to conclude (\ref{1pos2}) from (\ref{1pos}).

Now suppose $(M,g)$ is critical for $\tmf_{\tau}$ and let $g_1 = g + \theta_1$ be a metric which is sufficiently close in the $C^{2,\alpha}$-topology.  Let
$\tilde{\theta}$ satisfy (\ref{pullscale10})--(\ref{trconSn}).  In particular,
\begin{align*}
\tilde{\mathcal{F}}_{\tau}[g(Q) + \tilde{\theta}] &= \tilde{\mathcal{F}}_{\tau}[ e^c \phi^{*}g_1 ]
= \tilde{\mathcal{F}}_{\tau}[ \phi^{*}g_1 ] \quad \mbox{(by scale-invariance)} \\
&= \tilde{\mathcal{F}}_{\tau}[ g_1 ] = \tilde{\mathcal{F}}_{\tau}[g + \theta_1]
\quad \mbox{(by diffeomorphism-invariance)}.
\end{align*}

Define
\begin{align} \label{ladef}
a(s) = \tmf_{\tau} \bigg[g(Q) + s \frac{\tilde{\theta}}{\| \tilde{\theta}\|_{H^2}} \bigg],
\end{align}
where $\| \tilde{\theta} \|_{H^2} = \| \tilde{\theta}\|_{H^2(g)}$.  Then
\begin{align} \label{endpoint}
a(0) = \tmf_{\tau}[g(Q)] = \tmf_{\tau}[g], \quad \quad a(\| \tilde{\theta} \|_{H^2} ) = \tmf_{\tau}[g(Q) + \tilde{\theta}] = \tilde{\mathcal{F}}_{\tau}[g + \theta_1],
\end{align}
and $a'(0) = 0$ since $g(Q)$ is critical.  Also,
\begin{align} \label{d2a} \begin{split}
a''(0) &= \int \bigg\langle J_{g(Q)}\bigg(\frac{\tilde{\theta}}{\| \tilde{\theta} \|_{H^2}}\bigg), \frac{\tilde{\theta}}{\| \tilde{\theta} \|_{H^2}} \bigg \rangle\ dV_{g(Q)}  \\
&= \frac{1}{\| \tilde{\theta} \|_{H^2}^2} \int \big \langle J_{g(Q)} \tilde{\theta} ,\tilde{\theta} \big \rangle\ dV_{g(Q)}.
\end{split}
\end{align}
Using the formulas for the expansion of the curvature from Section \ref{KM} and the fact that the second variation is only second order in the metric (see Remark \ref{Jrem}), we can estimate
\begin{align*}
\int \big \langle J_{g(Q)} \tilde{\theta} ,\tilde{\theta} \big \rangle\ dV_{g(Q)} = \int \big \langle J_g \tilde{\theta} ,\tilde{\theta} \big \rangle\ dV_g + \{ \mbox{Error} \},
\end{align*}
where
\begin{align} \label{Eest} \begin{split}
\big| \{ \mbox{Error} \} &\big| \leq C(g) \big[ \| g(Q)-g\|_{C^{2}} + \| g(Q)-g\|_{C^{2}}^2 \big] \| \tilde{\theta} \|_{H^2}^2 \\
&\leq C\big( |Q| + |Q|^2 \big) \| \tilde{\theta} \|_{H^2}^2,
\end{split}
\end{align}
since $Q \mapsto g(Q)$ is smooth.  By (\ref{Qsmall}), if $\| \theta_1 \|_{C^{2,\alpha}}$ is sufficiently small, then
\begin{align} \label{Est}
\big| \{ \mbox{Error} \} \big| \leq C  \| \theta_1 \|_{C^{2,\alpha}} \cdot \| \tilde{\theta} \|_{H^2}^2.
\end{align}
Therefore,
\begin{align*}
a''(0) &= \frac{1}{\| \tilde{\theta} \|_{H^2}^2} \int \big \langle J_{g(Q)} \tilde{\theta} ,\tilde{\theta} \big \rangle\ dV_{g(Q)} \\
&= \frac{1}{\| \tilde{\theta} \|_{H^2}^2} \Big\{ \int \big \langle J_g \tilde{\theta} ,\tilde{\theta} \big \rangle\ dV_g  + O(\| \theta_1 \|_{C^{2,\alpha}}) \| \tilde{\theta}\|_{H^2}^2 \Big\} \\
&= \frac{1}{\| \tilde{\theta} \|_{H^2}^2} \int \big \langle J_g \tilde{\theta} ,\tilde{\theta} \big \rangle\ dV_g + O(\| \theta_1 \|_{C^{2,\alpha}}).
\end{align*}
Also,
\begin{align} \label{trcon2} \begin{split}
\int tr\ \tilde{\theta}\ dV &= 0,  \\
\int \langle u_{\nu}, \tilde{\theta} \rangle \  dV &= 0, \ \ 1 \leq \nu \leq d.
\end{split}
\end{align}
Since $g$ is stable, hence $H^2$-stable, (\ref{trcon2}) implies
\begin{align*}
\frac{1}{\| \tilde{\theta} \|_{H^2}^2} \int \big \langle J_g \tilde{\theta} ,\tilde{\theta} \big \rangle\ dV_g \geq \epsilon_0 > 0.
\end{align*}
Therefore, if $\| \theta_1 \|_{C^{2,\alpha}}$ is sufficiently small, then
\begin{align*}
a''(0) \geq \epsilon_0/2.
\end{align*}
Consequently, $a(s) > a(0)$ for all $s > 0$ sufficiently small (depending on $\epsilon_0$).  Therefore, if $\| \tilde{\theta}\|_{C^{2,\alpha}}$
is sufficiently small (which is guaranteed, via (\ref{tilsmall2}), if $\| \theta_1 \|_{C^{2,\alpha}}$ is
small enough), then
\begin{align*}
\tilde{\mathcal{F}}_{\tau}[g] = \tilde{\mathcal{F}}_{\tau}[g(Q)]< \tilde{\mathcal{F}}_{\tau}[g(Q) + \tilde{\theta}] = \tilde{\mathcal{F}}_{\tau}[g + \theta_1].
\end{align*}
In particular, if $\tilde{F}_{\tau}[g + \theta_1] = \tilde{\mathcal{F}}_{\tau}[g]$ then $\tilde{\theta} = 0$.  It follows that $g$ is a local minimum of $\tmf_{\tau}$.

If $g$ is strictly stable, then $\dim (H^1_{\tau}) = 0$ and in place of Lemma \ref{sliceSn} we can simply apply the slicing result of Lemma \ref{BasicSlice}, and the rest of the
proof goes through to show that $g$ is a strict local minimum of $\tmf_{\tau}$.

\end{proof}

Combining the results of Theorems \ref{TTPosCor} and \ref{conftPos} with the preceding proposition, we can now prove the main stability
result, Theorem \ref{StableMain} of the Introduction:

\begin{proof}[Proof of Theorem \ref{StableMain}.]
Suppose (\ref{stablegap}) holds.  By Theorems \ref{TTPosCor} and \ref{conftPos}, the Jacobi operator $J = J_g$ of $\tmf_{\tau}$ is positive
definite when restricted to either transverse-traceless or conformal variations.  Let $h$ satisfy $\beta_g h = 0$; we need to show
\begin{align*}
\int \big \langle Jh, h \rangle\ dV \geq \epsilon_0 \int |h|^2\ dV
\end{align*}
for some $\epsilon_0 > 0$.  Write $h = z + f g$, where $z$ is the trace-free part of $h$; then
\begin{align*}
0 = \beta_g h = \beta_g z + \beta_g (fg) = \delta_g z.
\end{align*}
Therefore, $z$ is transverse-traceless.  Then
\begin{align*}
\int \big \langle Jh, h \rangle\ dV  = \int \langle  J z,  z  \rangle dV + \int \langle  J (fg),  fg  \rangle dV + 2 \int \langle  J z,  fg  \rangle dV,
\end{align*}
where the last line follows from the fact that $J$ is self-adjoint (so we can combine the cross-terms).
However, the cross term vanishes
because $tr (Jz) = 0$ for $z$ transverse-traceless (as easily seen by taking a trace of \eqref{ttnvar}).  Therefore,
\begin{align*}
\int \big \langle Jh, h \rangle\ dV &=  \int \langle  J z,  z  \rangle dV + \int \langle  J (fg),  fg  \rangle dV \\
&\geq \epsilon_0 \int \big[ |z|^2 + f^2 \big]\ dV \geq \epsilon_0 \int |h|^2\ dV.
\end{align*}
We conclude that $g$ is strictly stable, so Theorem \ref{StableMain} follows from Proposition \ref{locmin}.
\end{proof}

In the case of the Bach tensor, using the terminology introduced in
 Definition~\ref{Bachdef}, we have the following
analogue to Proposition \ref{locmin}.
\begin{proposition}  \label{Bachlocmin}
Let $(M^4,g)$ be Bach-flat.
If $g$ is (strictly) Bach-stable, then $g$ is a (strict) local
minimizer for $\mathcal{W}$.
\end{proposition}
\begin{proof}
The details are similar to the proof of Proposition \ref{locmin},
with the following modification.
We introduce an extra conformal factor as in the slicing
in Theorem~\ref{PBzed}, to arrange moreover that $tr_{g_0}( \tilde{\theta}) = 0$
in Lemma~\ref{sliceSn}. Since the functional is conformally invariant, the argument in
the proof of Proposition~\ref{locmin} then extends to this case.
\end{proof}
Next, we prove the main theorem regarding Bach-flat metrics.
\begin{proof}[Proof of Theorem \ref{Bachthm}.]
The rigidity statement follows from Theorem~\ref{h1thm} and
Corollary~\ref{Kurcor}. For the local minimization statement,
strict Bach-stability follows from Theorems~\ref{TTPosCor} and~\ref{TTNegCor},
and the local minimization statement then follows from
Proposition \ref{Bachlocmin}.
\end{proof}

\subsection{Applications to reverse Bishop's inequalities}

One consequence of stability is a reverse Bishop's inequality,
Theorem \ref{RevBishop} of the Introduction:    \vskip.1in

\noindent {\em Proof of Theorem \ref{RevBishop}.}  Since $g$ is stable,
for all metrics $\tilde{g}$ in a $C^{2,\alpha}$-neighborhood of $g$ we have
\begin{align} \label{gmin1} \begin{split}
\tilde{\mathcal{F}}_0[\tilde{g}] &=
Vol(\tilde{g})^{4/n -1} \int |Ric(\tilde{g})|^2\ dV_{\tilde{g}} \\
&\geq \tmf_0[g] = Vol(g)^{4/n -1} \int |Ric(g)|^2\ dV_g = n(n-1)^2 Vol(g)^{4/n},
\end{split}
\end{align}
and equality holds if and only if $\tilde{g}$ is (up to scaling) isometric to $g$.
Assume $\tilde{g}$ satisfies
\begin{align*}
 Ric(\tilde{g}) \leq (n-1)\tilde{g}.
\end{align*}
An additional condition on the neighborhood $U$ is that $\tilde{g}$ also satisfies
\begin{align}
 Ric(\tilde{g}) > - (n-1)\tilde{g},
\end{align}
and consequently
\begin{align*}
|Ric(\tilde{g})|^2 \leq n(n-1)^2,
\end{align*}
hence
\begin{align} \label{vgp1}
\tilde{\mathcal{F}}_0[\tilde{g}] = Vol(\tilde{g})^{4/n -1} \int |Ric(\tilde{g})|^2\ dV \leq n(n-1)^2 Vol(\tilde{g})^{4/n}.
\end{align}
Combining this with (\ref{gmin1}) we get
\begin{align*}
Vol(\tilde{g}) \geq Vol(g).
\end{align*}
Moreover, if equality holds then $\tilde{g} = e^c \phi^{*}g$, but since $g$ and $\tilde{g}$ have
the same volume, $c = 0$.  This completes the proof.  \vskip.1in

A similar argument gives

\begin{corollary} \label{RicRigidNeg}  Let $(M,g)$ be a negative Einstein manifold, normalized so that
\begin{align} \label{normalpos2}
Ric(g) = -(n-1)g.
\end{align}
Assume
\begin{align}
\label{llc2}
\lambda_1(-\Delta_L) = \lambda_L > \frac{2}{n} R.
\end{align}

\noindent $(i)$  If $n =3$ or $4$, then there exists a $C^{2,\alpha}$-neighborhood $U$ of $g$ such that if $\tilde{g} \in U$ with
\begin{align} \label{gpgreater}
 Ric(\tilde{g}) \geq -(n-1)\tilde{g},
\end{align}
then
\begin{align} \label{Vcomp2}
Vol(\tilde{g}) \geq Vol(g)
\end{align}
and equality holds if and only if $\tilde{g}$ is isometric to $g$.

\noindent $(ii)$  If $n \geq 5$, assume in addition
\begin{align}
\lambda_1(-\Delta) > \frac{2(n-4)(n-1)}{n}.
\end{align}
Then the same conclusion holds.
\end{corollary}

\section{Examples} \label{ExampleSec}
In this section, we apply the above computations to
various examples and determine conditions for
rigidity and stability.

\subsection{ The round sphere $(S^n, g_S)$ (or any quotient thereof)}
\begin{proof}[Proof of Theorem \ref{snit}]
Normalizing so that $Ric(g_s) = (n-1) g_S$, the Lichnerowicz Laplacian
on TT-tensors is easily verified to be
\begin{align*}
\Delta_L h & =  \Delta h - 2n h.
\end{align*}
The inequality
\begin{align}
\int_{S^n} | \nabla_{i} h_{jk} + \nabla_j h_{ki} + \nabla_{k} h_{ij}|^2 dV \geq 0,
\end{align}
implies that the
least eigenvalue of the rough Laplacian on
TT-tensors is $2n$ \cite{Koiso1}.
Consequently, least eigenvalue of the Lichnerowicz Laplacian
on TT-tensors is
\begin{align}
4n = \frac{4}{n-1} R.
\end{align}
Theorem \ref{TTPosCor} then implies that if
\begin{align}
\label{upest}
\tau < \frac{2}{n(n-1)},
\end{align}
then the Jacobi operator is positive definite when restricted to
TT-tensors.
For the conformal direction, we can simply quote Theorem \ref{conftPos}
to obtain the lower restriction on $\tau$.
If $(M,g)$ is a proper quotient of $S^n$, then
the theorem follows from Proposition \ref{locmin}.
In the case of the round metric, for $\tau$ in the specified
range, $H^1_{\tau}$ consists of pure trace elements
of the form $\{ f g : f \in \Lambda \}$,
where $\Lambda$ is the space of first-order spherical harmonics.
These directions are clearly integrable, since they arise
from conformal diffeomorphisms.
Local minimization then follows from Proposition \ref{locmin}.
However, since all of the minimizing metrics are isometric,
we actually have strict local minimization.

The final statement in the theorem regarding $\tilde{\mathcal{R}}$
follows since $\tau = -1/(2(n-1))$ is always
included in the above ranges of $\tau$, and
since these metrics obviously minimize $\tilde{\mathcal{W}}$
(they are locally conformally flat), together with \eqref{RMF}.
\end{proof}
As a consequence of $\tau=0$,
we see that Theorem \ref{RevBishop} applies to any spherical space form.

\begin{remark}{\em
 We next briefly discuss the size of the neighborhood $U$ appearing in the
above theorem.  Given $\epsilon > 0$ small and
a constant $\Lambda$ large, there exists a $\delta > 0 $ such that
if
\begin{align}
\Vert Rm - (1/2) g \varowedge g \Vert_{C^0} < \delta,
\ \Vert Rm \Vert_{C^{\alpha}} < \Lambda,
\end{align}
then $\Vert g - g_S \Vert_{C^{2,\alpha}} < \epsilon$
where $g_S$ is a constant curvature metric \cite[Chapter 10]{Petersen}.
In the case of a space form, we suspect that the neighborhood of stability in
Theorem \ref{StableMain} can be weakened from pointwise $C^{2, \alpha}$-norm
to the integral condition
\begin{align}
\label{ln2}
\Vert Rm(\tilde{g}) - (1/2) \tilde{g} \varowedge \tilde{g} \Vert_{L^{n/2}} < \delta,
\end{align}
for $\delta$ sufficiently small.}
\end{remark}

We next consider the rigidity of space forms.

\begin{theorem} Let $(M,g)$ be a spherical space form,
and let $\tau$ be in the ranges specified in Lemma \ref{specgapLemma}.

\noindent $(i)$ If $(M,g)$ is not isometric to the sphere, then $H^1_{\tau}$
contains only traceless tensors.

\noindent $(ii)$ On $(S^n, g_S)$,  the pure trace elements in $H^1_{\tau}$
are exactly $\{ f g : f \in \Lambda \}$,
where $\Lambda$  is the space of first-order spherical harmonics.

In either case, if one orders the TT-spectrum of $(-\Delta_L)$ as
\begin{align}
0 < \mu_1 = 4n < \mu_2 < \mu_3 < \cdots,
\end{align}
then $H^1_{\tau}$ contains non-zero TT-tensors
exactly at the sequence of discrete values
\begin{align}
\tau_i = \frac{1}{2n(n-1)} \Big( \mu_i - 4(n-1) \Big).
\end{align}
If $n=4$ and $\tau = -1/3$, then $H^1_{\tau} = \{ 0 \}$, thus
any spherical space form is Bach-rigid.
\end{theorem}
\begin{proof}
This easily follows from Theorem \ref{h1thm} and Corollary~\ref{Kurcor}.
\end{proof}
As seen above in the proof Theorem \ref{snit}, on the round sphere,
for $\tau$ such that $H^1_{\tau} = \{ f g : f \in \Lambda \}$,
$g_S$ is indeed rigid, since these directions are tangent to
conformal diffeomorphisms, and the moduli in these directions are
all isometric.

\begin{proof}[Proof of Theorem \ref{nonit}.]
 As in \cite{Lamontagne}, we consider the Lie group $SU(2)$, and choose the following
basis of the Lie algebra $\mathfrak{su}(2)$
\begin{align*}
\left\{
\left(
\begin{matrix}
i & 0\\
0 & -i\\
\end{matrix}
\right) ,
\left(
\begin{matrix}
0 & i\\
i & 0 \\
\end{matrix}
\right) ,
\left(
\begin{matrix}
0 & -1\\
1 & 0 \\
\end{matrix}
\right)
\right\}.
\end{align*}
Let $\{e_1, e_2, e_3 \}$ denote the corresponding basis
of left-invariant vector fields. We declare these to
be orthogonal with $e_1$ and $e_2$ of length $1$,
and $e_3$ of length $s$ for positive $s \in \mathbf{R}$.
The resulting metric is known as a Berger sphere, and
for $s = 1$, is the round metric $g_S$.
It follows that $\{\tilde{e}_1, \tilde{e}_2, \tilde{e}_3 \}
= \{e_1, e_2, \frac{1}{s}e_3 \}$ is an orthonormal frame
satisfying the commutation relations
\begin{align*}
\newcommand{\te}{\tilde{e}}
[ \te_1, \te_2] = 2s \te_3,\
[ \te_2, \te_3] = \frac{2}{s} \te_1,\
[ \te_3, \te_1] = \frac{2}{s} \te_2.
\end{align*}
From these relations, it easily follows that in this basis
the Ricci tensor is diagonal with eigenvalues
$\{ 4 - 2s^2, 4-2s^2, 2s^2 \}$. We then have
\begin{align*}
|Ric|^2(s) &= 32 - 32s^2 + 12s^4,\\
R^2(s) &= 64 - 32s^2 + 4s^4,
\end{align*}
and therefore
\begin{align*}
|Ric|^2 + \tau R^2 = 32(1 + 2\tau) - 32(1 +\tau) s^2 + 4 (3 + \tau) s^4.
\end{align*}
Consequently,
\begin{align*}
\tmf_{\tau}[g_s] = c \tau^{4/3} (  32(1 + 2\tau) - 32(1 +\tau) s^2 + 4 (3 + \tau) s^4),
\end{align*}
for some constant $c > 0$.
At the value of $\tau = 1/3$, it is easily verified that
\begin{align}
\label{3der}
&\frac{d \tmf_{\tau}[g_s]}{ds}\Big|_{s = 1} = 0, \  \mathrm{ and }
\  \frac{d^2 \tmf_{\tau}[g_s]}{ds^2}\Big|_{s = 1} = 0, \ \mathrm{ but }
\ \frac{d^3 \tmf_{\tau}[g_s]}{ds^3}\Big|_{s = 1} > 0,
\end{align}
which shows that $s = 1$ is not a strict minimizer,
and therefore $ \tau = 1/3$ is sharp for stability

Finally, $h = g_s'$ is a multiple of $e_3 \otimes e_3$, so modulo a
constant multiple of the metric, is in the lowest eigenspace of
$(-\Delta_L)$, which is $H^1_{1/3}$. It is straightforward to see that
if there were {\em{any}} path of unit-volume critical
metrics $\tilde{g}_s$ with $\tilde{g}_s' = h$, then
\eqref{3der} would hold for  $\tilde{g}_s$. This is a contradiction
since the functional would be constant along  $\tilde{g}_s$.
\end{proof}
\begin{remark}{\em
For  $-1/2 < \tau < -1/3$, it is easy to see that that there is another
critical metric for some value $s < 1$, which is a local maximum in the
Berger-parameter direction. As in \cite{Lamontagne},
these are examples of non-Einstein critical metrics. }
\end{remark}

\subsection{Hyperbolic manifolds}
\begin{proof}[Proof of Theorem \ref{hypint}]

 Let $(M,g)$ be a compact hyperbolic manifold in dimension $n$,
satisfying $Ric(g) = -(n-1) g$.
The Lichnerowicz Laplacian on TT-tensors is
\begin{align*}
\Delta_L h & =  \Delta h + 2n h.
\end{align*}
The inequality
\begin{align}
\label{codin}
\int_M  | \nabla_i h_{jk} - \nabla_j h_{ik} |^2 dV \geq 0,
\end{align}
implies that the least eigenvalue of the rough Laplacian on
TT-tensors is $n$, with equality for
Codazzi tensors \cite{Koiso1}. Consequently,
on a hyperbolic manifold,
the least eigenvalue of the rough Laplacian
on TT-tensors is bounded below by
\begin{align}
- n = \frac{1}{n-1} R.
\end{align}
Theorem \ref{TTNegCor} implies that $\tmf_{\tau}$ is strictly
stable in the TT-direction for
\begin{align}
\tau > \frac{4 - 3n}{2 n (n-1)}.
\end{align}
Combining this with Theorem \ref{conftNeg} for
strict stability in the conformal direction,
this proves strict stability. Strict
local minimization then follows from Proposition \ref{locmin}.
The local minimization statement for $\tilde{\mathcal{R}}$
for $n = 3$ and $n = 4$ follows as in the spherical case.
\end{proof}

\begin{remark}{\em
For $n \geq 5$, the upper endpoint may be improved given an
improved estimate on the lowest eigenvalue of the Laplacian
on functions. It is easy to show that if
\begin{align}
\lambda_1 >  \frac{2(n-4)}{n^2} (-R),
\end{align}
then the interval may be extended to include $\tau = 0$. }
\end{remark}

We next consider the case of rigidity of compact hyperbolic manifolds.
\begin{theorem} Let $(M,g)$ be a compact hyperbolic manifold,
and let $\tau$ be in the ranges specified in Lemma \ref{specgapLemma}.
Then $H^1_{\tau}$ contains no pure-trace elements.
If one orders the TT-spectrum of $(-\Delta_L)$ as
\begin{align}
-n \leq \mu_1  < \mu_2 < \mu_3 < \cdots,
\end{align}
then $H^1_{\tau}$ contains non-zero TT-tensors
only at the sequence of discrete values
\begin{align}
\tau_i = - \frac{1}{2n(n-1)} \Big( \mu_i + 4(n-1) \Big).
\end{align}
\end{theorem}
\begin{proof}
This easily follows from Theorem \ref{h1thm}.
\end{proof}

\begin{remark}{\em
In the case of $n =4$, $\tau = -1/3$, a non-trivial TT-deformation
yields equality in \eqref{codin}, which means that
$h$ is a Codazzi tensor. Such deformations are exactly
locally conformally flat deformations, and in this situation
the hyperbolic structure is called {\em{bendable}}
\cite{Lafontaine, JohnsonMillson}.}
\end{remark}

\subsection{Complex projective space $\CP^m$}
\begin{proof}[Proof of Theorem \ref{cpmint}]

We begin with an eigenvalue estimate on $\CP^m$ proved using
representation theory in the case $m = 2$ in \cite{Warner}, and for all
$m$ by direct calculation in \cite{Boucetta}.
\begin{proposition}[\cite{Boucetta, Warner}]
\label{evcp}
On $(\CP^m, g_{FS})$, with $Ric = 2(m+1) g$,
the least eigenvalue of the Lichnerowicz Laplacian
on TT-tensors is
\begin{align}
8(m+2) = \frac{2(m+2)}{m(m+1)} R.
\end{align}
\end{proposition}
\noindent
Theorem \ref{TTPosCor} then implies that if
\begin{align}
\label{upestcp}
\tau < \frac{1}{m(m+1)},
\end{align}
then $\tmf_{\tau}$ is strictly stable in the
TT-direction.
For the conformal direction, we can simply quote Theorem \ref{conftPos}.
This proves strict stability, and strict
local minimization then follows from Proposition \ref{locmin}.
The statement for $\mathcal{R}$ in dimension $4$ follows for the reasons as discussed
in the introduction.
\end{proof}
\noindent
As a consequence of $\tau=0$, we see that Theorem \ref{RevBishop} applies to
to $(\CP^m, g_{FS})$.

We next consider the case of rigidity in the case of $m =2$.

\begin{theorem} On $\CP^2$, the Fubini-Study metric $g_{FS}$ is infinitesimally rigid
(and thus rigid) provided that $\tau < 1/6$. If one orders the TT-spectrum of $(-\Delta_L)$ as
\begin{align}
0 < \mu_1 = 32 < \mu_2 < \mu_3 < \cdots,
\end{align}
then $H^1_{\tau}$ is nontrivial only at the sequence of discrete values
\begin{align}
\tau_i = \frac{1}{48} \Big( \mu_i - 24 \Big).
\end{align}
For $\tau = -1/3$, $H^1_{\tau} = 0$, thus $g_{FS}$ is Bach-rigid.
\end{theorem}
\begin{proof}
This easily follows from Theorem~\ref{h1thm} and Corollary~\ref{Kurcor}.
\end{proof}
As in the case of the sphere, we conjecture that the infinitesimal
deformations at the value $\tau_i$ are not integrable.

\subsection{The product metric on $S^m \times S^m$}
\label{pmsm}
In this subsection, we will discuss the proof of Theorem \ref{smsmint}.
\begin{proposition}
\label{evs2}
On $(S^m \times S^m, g_1 + g_2)$, with $Ric = (m-1) g$,
the least eigenvalue of the Lichnerowicz Laplacian
on TT-tensors is $0$, with $1$-dimensional
kernel spanned by $g_1 - g_2$.  The next eigenvalue is greater than
or equal to $2m$.
\end{proposition}
\begin{proof}
As in \cite{Kobayashi}, the general traceless symmetric tensor splits as
\begin{align}
h = \overset{\circ}{h_1}  + \frac{f}{m} g_1
  + \hat{h} + \overset{\circ}{h_2} - \frac{f}{m} g_2,
\end{align}
where $h_1$ is tangent to the first factor, $\overset{\circ}{h_1}$ is its trace-free part,
$h_2$ is tangent to the second factor, and
$\hat{h}$ are the mixed directions.
The above decomposition is orthogonal,
and is preserved by $\Delta_L$.
The curvature tensor is given by
\begin{align}
R_{ijkl} =  (g_1)_{ik} (g_1)_{jl} -  (g_1)_{jk} (g_1)_{il}
+  (g_2)_{ik} (g_2)_{jl} -  (g_2)_{jk} (g_2)_{il},
\end{align}
with $Ric = (m-1)g$, and $R = 2 m (m-1)$.
If $h$ satisfies $(-\Delta_L) h = \lambda h$,
then we have the separate equations
\begin{align}
\label{lq1}
\Delta   \overset{\circ}{h_i} - 2m  \overset{\circ}{h_i}
&= (- \lambda ) \overset{\circ}{h_i}   \\
\label{lq2}
\Delta f &= (-\lambda) f\\
\label{lq3}
\Delta \hat{h}  - 2(m-1) \hat{h} &= (-\lambda) \hat{h}.
\end{align}

\begin{lemma} \label{DLSn} The lowest eigenvalue of $(-\Delta_L)$ acting
on tensors of the form  $h = \overset{\circ}{h_1}$ or
$h = \overset{\circ}{h_2}$ is at least $4m$ if $h$ is divergence-free.
If $h$ is not divergence-free, then the smallest eigenvalue is
at least $2(m+1)$.
\end{lemma}

\begin{proof}
Without loss of generality, assume $h = \overset{\circ}{h_1}$.
We first assume that $h$ is divergence-free.  Next
symmetrize and consider
\begin{align}
\begin{split}
&\int | \nabla_i h_{jk} + \nabla_j h_{ki} + \nabla_k h_{ij}|^2 dV
= 3 \int |\nabla h|^2 dV + 6 \int \nabla_i h_{jk} \nabla_j h_{ki} dV\\
& = 3 \int |\nabla h|^2 dV - 6  \int h_{jk} \nabla_i \nabla_j h_{ki} dV\\
& =  3 \int |\nabla h|^2 dV - 6 (m-1)\int |h|^2 dV
+  6 \int h_{jk} R_{ijk}^{\ \ \ p} h_{pi} dV.
\end{split}
\end{align}
Thus we have the inequality
\begin{align}
\label{nhin0}
\int |\nabla h|^2 dV \geq 2(m-1) \int |h|^2 dV
-  2 \int h_{jk} R_{ijk}^{\ \ \ p} h_{pi} dV.
\end{align}
Since $h$ is ``tangent'' to the first $S^m$-factor and $h$ is traceless,
the last integral is
\begin{align*}
-  2 \int h_{jk} R_{ijk}^{\ \ \ p} h_{pi} dV =  2 \int |h|^2 dV.
\end{align*}
Substituting this into \eqref{nhin0}, we obtain
\begin{align}
\label{nhin2}
\int |\nabla h|^2 dV \geq 2(m-1) \int |h|^2 dV + 2 \int|h|^2 dV
= 2m \int |h|^2 dV.
\end{align}
From \eqref{lq1}, the least eigenvalue of $(-\Delta_L)$ on divergence-free
is therefore at least  $2m + 2m = 4m$.

Next if $h$ is not divergence-free, we use the formula
\cite{Lichnerowicz}
\begin{align}
\delta (\Delta_L h) = - \Delta_H ( \delta h).
\end{align}
So if $h$ is an eigenfunction of $(-\Delta_L)$
with eigenvalue $\nu$, then $\delta h$ is a
eigenfunction of $\Delta_H$ with the same eigenvalue.

From Folland \cite{Folland},
the eigenspaces of $\Delta_H$ on $1$-forms on $S^m$
are spanned by the $1$-forms
$\phi_{1\ell}$ with eigenvalues $(\ell+1)(\ell + m -2)$, and
$\psi_{1\ell}$ with eigenvalues $\ell (\ell + m -1)$.
Note that $\delta \phi_{1\ell} = 0$, and $d \psi_{1\ell} = 0$.
The first few eigenvalues of $\phi_{1\ell}$ are
\begin{align}
\label{pe2}
2(m-1), \ 3m, \ 4(m+1), \dots.
\end{align}
These are all of the divergence-free eigenforms.
Note that the $2(m-1)$-eigenvalue here corresponds to Killing forms.
The first few eigenvalues of $\psi_{1\ell}$ are
\begin{align}
\label{se2}
m, \ 2(m+1), \ 3(m +2), \dots.
\end{align}
This list is exactly the eigenvalues of the Laplacian on functions
(excepting the zero eigenvalue); this is not an accident since
$\psi_{1\ell} = d f_{\ell}$ where $f_{\ell}$ is a harmonic function with the
same eigenvalue, so none of these are divergence-free.

The formula
\begin{align}
\Delta_L h &= \Delta h - 2m h
\end{align}
tells us that the least eigenvalue of $(-\Delta_L)$
is at least $2m$, with equality if and only if $\Delta h = 0$,
in which case $h$ is parallel. Then $h$ is necessarily divergence-free,
and the above argument says the least eigenvalue of $\Delta$ on
divergence-free tensors is at least $4m$, so this cannot happen.
Consequently, the least eigenvalue is strictly greater than $2m$.

Next, the argument in \cite{Bergeretal} says that the eigenspace of the
Hodge Laplacian on $1$-forms with eigenvalue $\mu$ is spanned by forms of the form
\begin{align}
f_1(z_1) \alpha_2 (z_2) , \ f_2(z_2) \alpha_1(z_1),
\end{align}
where $f_i$ is an eigenfunction of the Laplacian on the
$i$-th $S^m$-factor with eigenvalue $\nu$, and $\alpha_i$ is an eigen-$1$-form
of the Hodge Laplacian $\tilde{\nu}$ for  $i = 1,2$ with
$\mu = \nu + \tilde{\nu}$.
Consequently, the list of eigenvalues on $1$-forms on the product is
given by all sums from \eqref{pe2} and the second list \eqref{se2}
taken with the second list beginning with $0$.
Therefore the list of eigenvalues on $1$-forms on the product begins with
\begin{align}
m, 2(m-1), 2m, 2(m+1), 3m -2, 3m, \dots
\end{align}
We already know that $\nu > 2n$, so
the next possiblity is $2(m + 1)$.

\end{proof}

The preceding lemma says that for any eigenvalue strictly less
than $2(m+1)$, we must have $\overset{\circ}{h_i} = 0$
for $i = 1,2$. So we next consider $h$ of the form
\begin{align}
\label{sform}
h = \frac{f}{m} g_1 + \hat{h} - \frac{f}{m} g_2.
\end{align}
In this case, we have the two equations
\begin{align}
\label{evq2}
\Delta f &= (-\lambda) f\\
\label{evq3}
\Delta \hat{h}  - 2(m-1) \hat{h} &= (-\lambda) \hat{h}.
\end{align}
We recall the eigenvalues of the Laplacian on functions on $S^m$ are
\begin{align}
\label{se22}
0, m, \ 2(m+1), \ 3(m +2), \dots.
\end{align}

The argument in \cite{Bergeretal} says that the eigenspace of the
Laplacian on with eigenvalue $\mu$ is spanned by functions of the form
\begin{align}
f_1(z_1) f_2 (z_2),
\end{align}
where $f_i$ is an eigenfunction of the Laplacian on the
$i$-th factor with eigenvalue $\nu$ for  $i = 1,2$ with
$\mu = \nu + \tilde{\nu}$.
Consequently, the list of eigenvalues on functions on the product is
given by all sums from the list \eqref{se22} with itself, which
begins with
\begin{align}
0, m, 2m, 2(m + 1), \dots.
\end{align}

\begin{lemma}
\label{hathp}
For divergence-free $h$ of the form $ h = \hat{h}$,
we have
\begin{align}
\int |\nabla \hat{h}|^2 \geq 2 (m-1) \int |\hat{h}|^2.
\end{align}
\end{lemma}

\begin{proof}
As in the proof of Lemma \ref{DLSn}, we symmetrize to obtain
\begin{align}
\label{nhin}
\int |\nabla h|^2 dV \geq 2 (m-1) \int |h|^2 dV -  2
\int h_{jk} R_{ijk}^{\ \ \ p} h_{pi} dV.
\end{align}
Since the curvature tensor has no mixed terms, the last term
vanishes, and the result follows.
\end{proof}

Continuing with the argument,
first, assume that $f$ is non-zero.
The first possibility is $\lambda = 0$, in which case
$f$ is constant.  Since $h$ is divergence-free, it follows that $\delta \hat{h} = 0$, and
the Proposition then implies that $\hat{h} = 0$.
This case does in fact happen, since $g_1 - g_2$ is a
kernel element.

Still assuming $f$ is non-zero,
the next possibility is $\lambda = m$. In this case $f$ is
either $f(z_1)$ or $f(z_2)$, where $f$ is a first eigenfunction
on $S^m$. Without loss of generality, assume that
$f = f(z_1)$. Then if $\hat{h}$ is non-zero, the divergence-free
condition on $h$ yields
\begin{align}
\delta \hat{h} =- \frac{1}{m} df.
\end{align}
Integrating, we obtain
\begin{align}
\begin{split}
\int \langle df, df \rangle dV
&= -m \int \langle df , \delta  \hat{h} \rangle dV
= m \int \langle \nabla^2 f,  \hat{h} \rangle dV = 0,
\end{split}
\end{align}
since the hessian of $f$ only lives on the first factor,
and $\hat{h}$ has only mixed components.
This implies $df = 0$, which is a contradiction since $f$ is
not-constant. Therefore $\hat{h} = 0$, which again says $df = 0$,
which is a contradiction.

In the above 2 paragraphs, we have shown that if $f$ is non-zero,
then $\lambda \geq 2m$ unless $\lambda = 0$, and $f$ is constant.
So we next consider the case that $\lambda < 2m$.
If $f$ is a non-zero constant, then $\lambda = 0$.
Then $\delta \hat{h} = 0$, and
Lemma \ref{hathp} then implies that $\hat{h} = 0$.
This proves the zero eigenspace is spanned by $g_1 - g_2$.
Next, if $\lambda < 2m$ and $f$ is not constant, then $f$
must vanish. In this case, $\delta \hat{h} = 0$,
and the proposition says $\lambda \geq 4(m-1)$, which
is a contradiction.

 Therefore we have shown that the kernel of $\Delta_L$
on traceless divergence-free tensors is spanned by $g_1 - g_2$,
and that the next largest eigenvalue is at least $2m$.
\end{proof}

Next, we show that for $m =2$, then $4$ is actually an eigenvalue.
We let
\begin{align}
\hat{h} &=   \alpha_1 \odot \alpha_2,
\end{align}
where $\alpha_i$ is a divergence-free eigenform of the rough Laplacian
on $1$-forms on the $i$-th factor, $i = 1, 2$. Here $\odot$
denotes symmetric product.  We have seen above
that the lowest eigenvalue of the rough Laplacian
on $S^2$ on divergence-free
$1$-forms is $1$. We then have
\begin{align}
\begin{split}
\Delta (   \alpha_1 \odot \alpha_2 )
&= \Delta (  \alpha_1 \otimes \alpha_2 +  \alpha_2 \otimes \alpha_1 )\\
& =  \Delta( \alpha_1) \otimes \alpha_2  + \alpha_1 \otimes  \Delta( \alpha_2)
 +  \Delta( \alpha_2) \otimes \alpha_1 + \alpha_2 \otimes \Delta(\alpha_1)\\
& = - 2 \alpha_1  \otimes \alpha_2  -2  \alpha_2 \otimes \alpha_1
= - 2 (\alpha_1 \odot \alpha_2),
\end{split}
\end{align}
and $\alpha_1 \odot \alpha_2$ is clearly traceless and divergence-free,
with \eqref{lq3} showing that the corresponding eigenvalue of $\Delta_L$ is
$4$.

\begin{remark}{\em
In higher dimensions, Lemma \ref{hathp} shows that the least eigenvalue of
$(-\Delta_L)$ on divergence-free $h$ of the form $\hat{h}$ is $4(m-1)$.
This does in fact occur as an eigenvalue by considering
$\alpha_1 \odot \alpha_2$ as we did above for $m = 2$.}
\end{remark}

Noting that
\begin{align}
2m = \frac{1}{m-1} R,
\end{align}
Theorem \ref{TTPosCor} then implies that if
\begin{align}
\label{upest3}
\tau < \frac{2-m}{2 m(m-1)},
\end{align}
then $\tmf_{\tau}$ is strictly stable in the
TT-direction.
For the conformal direction, we can simply quote Theorem \ref{conftPos}.
Strict local minimization then follows from Proposition \ref{locmin}, and
the statement for $\mathcal{R}$ in dimension $4$
follows for the reasons as discussed in the introduction.
this completes the proof of Theorem \ref{smsmint}.

We next consider the case of rigidity for $S^2 \times S^2$.
\begin{theorem} On $S^2 \times S^2$, the product metric
$g_1 + g_2$ is infinitesimally rigid (and therefore rigid)
provided that $\tau < - 1/2$.
If one orders the TT-spectrum of $(-\Delta_L)$ as
\begin{align}
\mu_1 = 0 < \mu_2 = 4 < \mu_3 < \cdots,
\end{align}
then $H^1_{\tau}$ is nontrivial only at the sequence of discrete values
\begin{align}
\tau_i = \frac{1}{8} \Big( \mu_i - 4 \Big).
\end{align}
For $\tau_1 = - (1/2)$, $\dim(H^1_{-1/2}) =1$, and
the path of metrics $s \mapsto e^s g_1 + e^{-s} g_2$ yields a
corresponding non-trivial
deformation, with $h = g_1 - g_2$. This
is a K\"ahler constant scalar curvature deformation, which
is a path of critical metrics for the functional
$\tmf_{\tau_1}$, therefore these infinitesimal
deformations are integrable.

 For $\tau_2 = 0$, $\dim ( H^1_0) = 9$, and this space is
spanned by tensors of the form $\alpha_1 \odot \alpha_2$
where $\alpha_i$ is a Killing form on the $i$-th factor, $i = 1, 2$.

For $\tau = -1/3$, $H^1_{\tau} = 0$, thus $g_1 + g_2$ is Bach-rigid \cite{Kobayashi}.
\end{theorem}
\begin{proof}
The follows from Theorem \ref{h1thm} and Corollary \ref{Kurcor}.
The fact that $H^1_{-1/2}$ is spanned $g_1 - g_2$ follows from Proposition \ref{evs2}.
The fact that $g_s$ are critical for $\tmf_{\tau_1}$ follows from \cite[Equation 38]{Derdzinski}, which says that
the Bach tensor of K\"ahler constant scalar curvature surface is a constant multiple
of the traceless Ricci tensor, which is exactly the Euler-Lagrange equation
for $\tau_1 = -1/2$. The determination of $H^1_0$ was carried out above.
\end{proof}
As in the previous examples,
we conjecture that the infinitesimal deformations at value $\tau_i$
for $i \geq 2$ are not integrable.

\subsection{The Ricci-flat case}
\begin{proof}[Proof of Theorem \ref{rfint}]
Theorem \ref{FtJf} shows that the Jacobi operator on pure trace
elements is
\begin{align}
tr(Jf) =  (n-4\tau + 4n\tau)(n-1) \Delta^2 f.
\end{align}
Consequently, $g$ is strictly stable in the conformal direction
for $-n/4(n-1) < \tau$.

Theorem \ref{JTT} shows that the Jacobi operator on TT-tensors
is a positive multiple of $\Delta_L^2$, and is therefore non-negative.
Integrating by parts, a TT-kernel element $h$ satisfies
$\Delta_L h = 0$, and is therefore an infinitesimal Einstein deformation.
Integrability of infinitesimal Einstein deformations on the torus is elementary, and the
integrability of such deformations on a Calabi-Yau manifold
was proved by Bogomolov and Tian \cite{Bogomolov, TianCY}.
Theorem \ref{rfint} then follows from Proposition \ref{locmin}.

\end{proof}
\bibliography{Rigidity_Stability}

\end{document}